\documentclass{amsart}
\usepackage[utf8]{inputenc}

\title[The Integrability of a knife-edge Billiard in a Disk]{The Integrability of a knife-edge Billiard in a Disk}

\author[A.\ Bravo-Doddoli]{Alejandro\ Bravo-Doddoli} 
\address{Alejandro Bravo-Doddoli: Department of Mathematics, University of Michigan, Ann Arbor, MI 48109, U.S.}
\email{Abravodo@umich.edu}
\author[H.\ Hou]{Huaidian\ Hou}
\address{Huaidian Hou: Computer Science and Engineering Division, University of Michigan, Ann Arbor, MI 48109, U.S.}
\email{houhd@umich.edu} 
\author[W. \ Clark]{William\ Clark}
\address{William Clark: Department of Mathematics, Ohio University, Athens, OH 45701, U.S.}
\email{clarkw3@ohio.edu}
\author[A. M. Bloch]{
Anthony M.\ Bloch}
\address{Anthony Bloch: Department of Mathematics, University of Michigan, Ann Arbor, MI 48109, U.S.}
\email{\href{abloch@umich.edu}{abloch@umich.edu}}

\date{June 2026}

\usepackage[english]{babel}
\usepackage[utf8]{inputenc}
\usepackage{libertine}
\usepackage[final]{graphicx}
\usepackage{floatflt}
\usepackage{blindtext}
\usepackage{enumitem}
\usepackage{amsthm}
\usepackage{subfig}
\usepackage{listings}
\usepackage{listingsutf8}
\usepackage{amsmath}
\usepackage{framed}
\usepackage{minibox}
\usepackage{float}
\usepackage{wrapfig}
\usepackage{longtable}
\usepackage[strict]{changepage}
\usepackage{pgfplots}
\usepackage{nicefrac}
\usepackage{units}
\usepackage{etoolbox,tikz}
\usepackage{bbm}
\usepackage{hyperref}
\usepackage[autostyle]{csquotes}
\usepackage{caption}
\usepackage{subcaption}

\usetikzlibrary{external}
\usetikzlibrary{cd}

\AtBeginEnvironment{tikzcd}{\tikzexternaldisable}
\AtEndEnvironment{tikzcd}{\tikzexternalenable}

\usepackage{amsmath}
\usepackage{amsfonts}

\usepackage{amsthm}
\usepackage {amsmath, amssymb}
\usepackage{pb-diagram}
\usepackage{tikz-cd}
\usepackage{mathtools}
\usepackage{amssymb}

\usetikzlibrary{matrix}
\pgfplotsset{width=11cm,compat=1.9}
\usepgfplotslibrary{external}
\DeclareMathOperator{\sgn}{sgn}
 
\def\R{\mathbb{R}}

\def\D{\mathcal{D}}
\def\Rm{\mathcal{R}}

\def\I{\mathcal{I}}
\def\T{\mathrm{T}}

\numberwithin{equation}{section}

\newtheorem{thm}{Theorem}
\newtheorem{proposition}[thm]{Proposition}
\newtheorem{corollary}[thm]{Corollary}
\newtheorem{lemma}[thm]{Lemma}
\newtheorem{definition}[thm]{Definition}

\newtheorem{remark}[thm]{Remark}

\renewcommand{\emph}[1]{{\bfseries\itshape{#1}}}

\usepackage{hyperref}
\usepackage[%
	capitalize,nameinlink
]{cleveref}

\hypersetup{%
	colorlinks=true,
	linkcolor=blue,
	citecolor=magenta,
	urlcolor=magenta,
}

\numberwithin{figure}{section}

\newcommand{\La}{\mathcal{L}}
\newcommand{\Si}{\mathcal{S}}
\newcommand{\horg}{\nabla_{\mathrm{hor}}}

\begin{document}

 \begin{abstract} 
This paper proves the integrability of a nonholonomic billiard defined by a knife-edge in a disk. The paper begins by parametrizing the impact space using coordinates that reveal the system’s rotational symmetry. It then constructs a map on the space of post-impact states that encodes the billiard dynamics through a recurrence relating each post-impact state to the next.  In addition, the paper shows that the knife-edge billiard admits a family of generalized caustics. 
 \end{abstract}

\maketitle

\section{Introduction}

The theory of classical billiards is a fruitful research area involving the study of Hamiltonian dynamics, integrability, and the transition to chaos \cite{tabachnikov2005geometry,kozlov1991billiards}. In recent years, the area has been extended through generalizations that modify the free-motion law or reflection law. In \cite{clark2019bouncing}, the third and fourth authors of this paper established a framework for studying nonholonomic billiards.
Our main result establishes the integrability of the nonholonomic billiard consisting of a knife-edge in a disk. A knife-edge is a classical nonholonomic system consisting of a planar rigid body  \cite[Chapter 1.6]{tony}. 


Before presenting our results in detail, we recall that a billiard system is a hybrid dynamical system: the continuous dynamics govern motion within a domain, while the discrete transitions are determined by an impact map on the boundary \cite{teel2012hybrid}. More precisely, a nonholonomic hybrid system consists of a quadruple $(Q,\mathcal{L},\D,\mathcal{S})$, where $Q$ is the configuration space, $\mathcal{L}:TQ\to \R$ is a smooth Lagrangian governing the smooth dynamics, $\D \subset TQ$ is called the constraint distribution, and $\mathcal{S}\subset Q$ is the impact surface; see Subsection \ref{ap:hyb-sys} for details.


The nonholonomic impact map 
$$\mathrm{I} :\D|_{\Si} \to \D|_{\Si}$$
is the rule that determines how the admissible velocity changes after an impact; that is, the impact map takes the pre-impact state $(q,\dot{q}^-)$ and sends it to the post-impact state 
\begin{equation*}
  \mathrm{I} (q, \dot{q}^-)= (q, \dot{q}^+), \text{ where } \dot{q}^-,\dot{q}^+\in \D_q .  
\end{equation*}
We call the impact map elastic if it conserves energy; see Subsection \ref{subsec:ela-non-hol-imp} for details.

\begin{figure}
    \centering
    \includegraphics[width=0.28\linewidth]{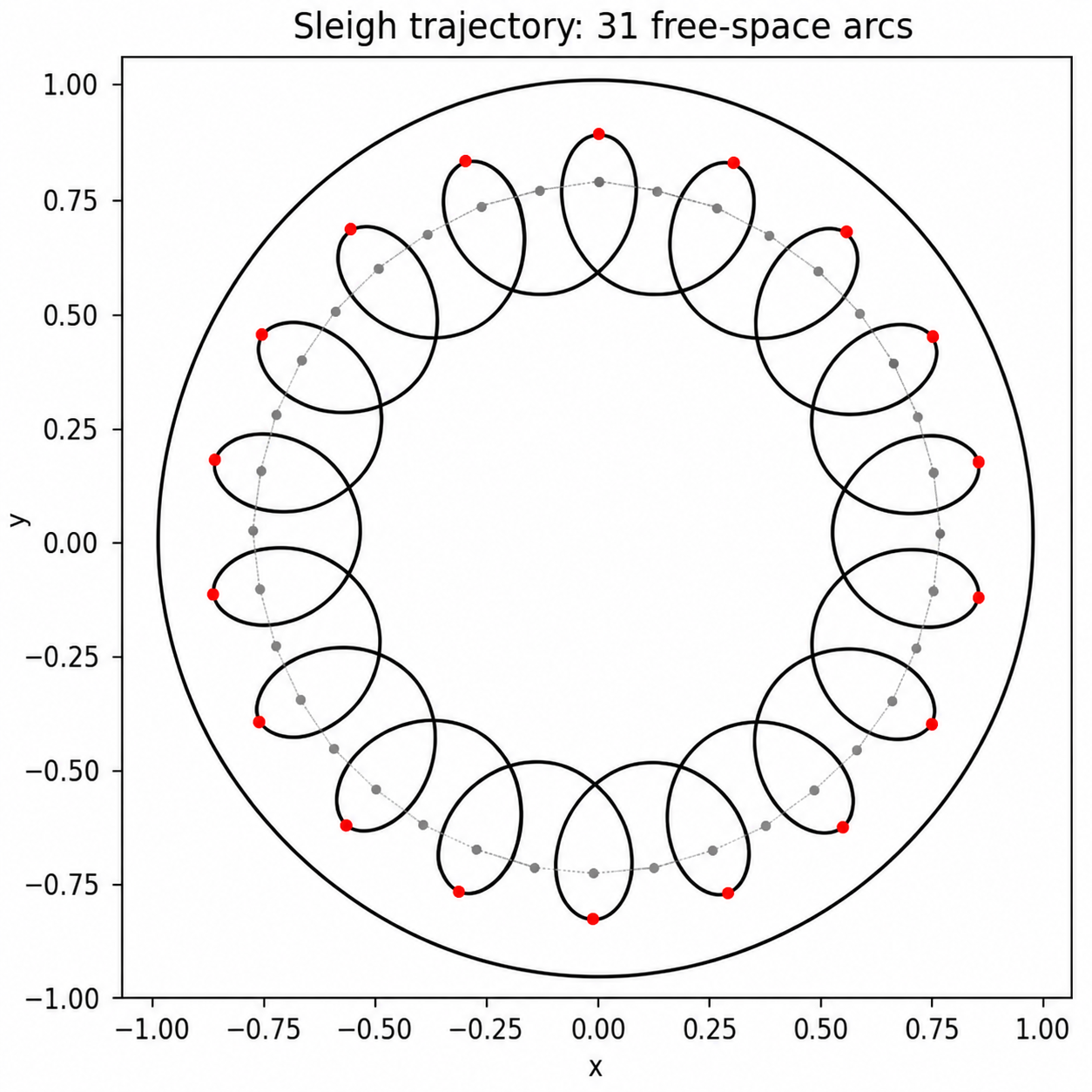}\quad \includegraphics[width=0.28\linewidth]{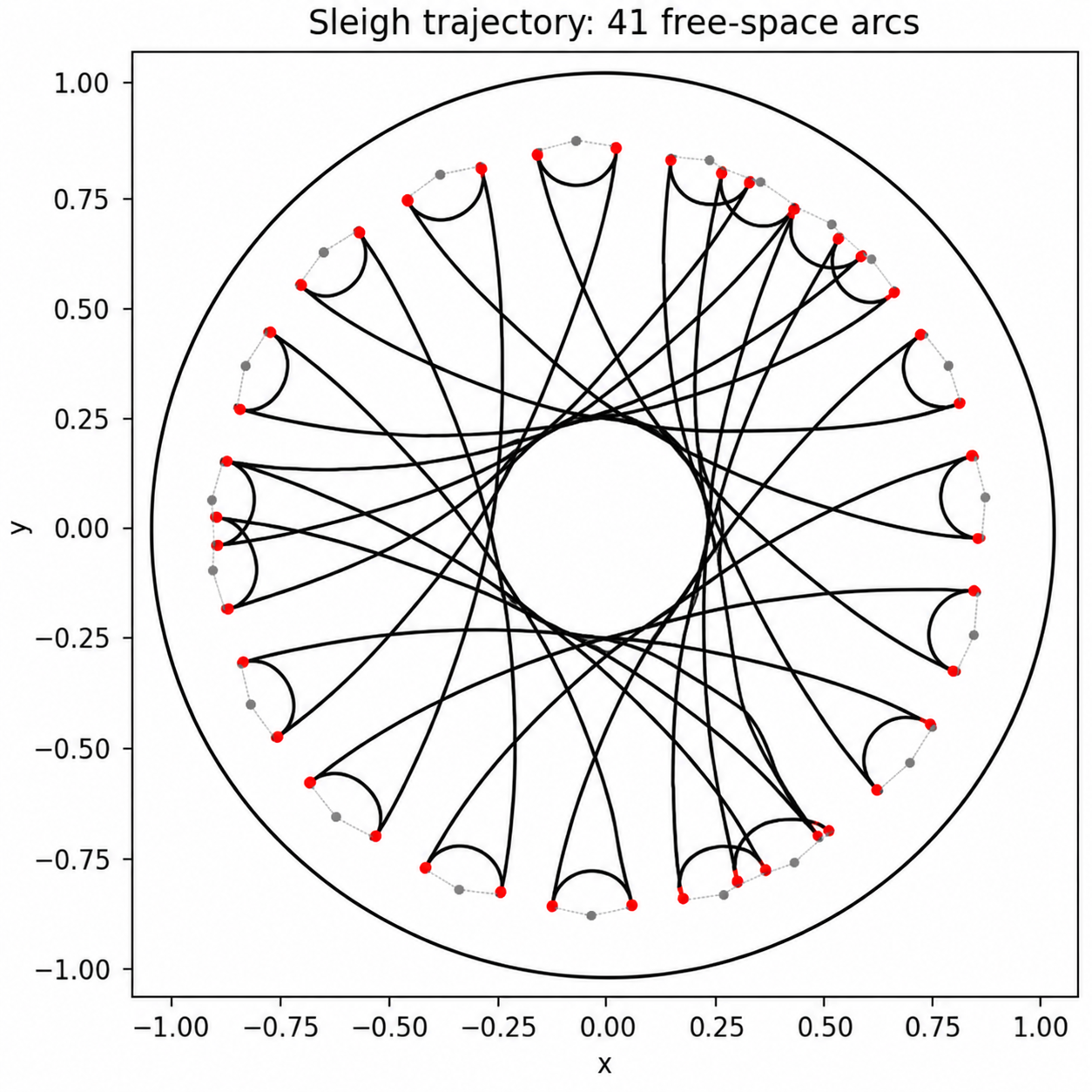}
    \caption{Examples of periodic and dense orbits for generic dynamics. }
    \label{fig:genric-mot}
\end{figure}

We model the knife-edge as a one-dimensional rigid body: a straight line segment of length $2\ell$, with total mass $m$, moment of inertia $J$ about the midpoint, and center of mass at the midpoint. Thus, the configuration space is $SE(2)$. The knife-edge imposes a nonholonomic constraint that prevents motion perpendicular to the blade. This constraint defines the standard left-invariant contact distribution $\D \subset TSE(2)$, which is spanned by the left-invariant vector fields $\{X_1,X_2\}$; see Eq. \eqref{eq:lef-in-SE2}. The quasi-velocities $(\omega,v )\in \R^2$ associated with this frame are the angular velocity $\omega$ of the knife-edge about the midpoint and the admissible linear velocity $v$ of the body. Since there is a linear bijection between a constraint velocity $\dot{q}\in \D_q$ and the pair $(\omega,v )\in\R^2$, we represent a state of the system $(q,\dot{q})$ by   $(q,\omega,v) \in SE(2)\times\R^2$. Throughout the paper, we refer to the hybrid dynamical system $(SE(2),\mathcal{L},\D,\mathcal{S})$ with elastic impacts as the \textbf{knife-edge billiard in the disk}.  Here, the Lagrangian function $\mathcal{L}$, the constraint distribution $\D$, and the impact surface $\mathcal{S}$ are defined by Eqs. \eqref{eq:lag-func}, \eqref{eq:cost-dis}, and \eqref{eq:round table-1}, respectively. 

Let $\mathbb{T}^2$ denote the two-torus.\footnote{The two-torus $\mathbb{T}^2$ is given by $\mathbb{S}^1\times \mathbb{S}^1$, where $\mathbb{S}^1:=\R/2\pi\mathbb{Z}$ denotes the quotient of $\R$ by the relation $\phi \sim\phi +2\pi k$, $k\in \mathbb{Z}$. The equivalence class of $\phi$ is denoted by $[\phi]$.} We identify the spaces of pre-impact and post-impact states with the subsets
$$\I^-\subset \mathbb{T}^2\times\R^2,\;\text{and}\;\I^+\subset \mathbb{T}^2\times\R^2,$$
respectively. A point $z^\pm\in \I^\pm$ has the form $z^\pm \coloneqq ([\phi],[\vartheta],\omega,v)$, where $[\phi] \in \mathbb{S}^1$  parametrizes the boundary of the disk and $[\vartheta] \in \mathbb{S}^1$ describes the relative angle between the outward normal vector to the disk boundary and the orientation of the knife-edge. The angle $[\vartheta]$ is relevant because it is invariant under the action of $SO(2)$; see Subsection \ref{subsec:set-bil}. 

Next, we introduce the \textbf{coupled circle map}\footnote{This terminology reflects that $\mathrm{Ro}$ acts as a coupled circle map on the angular variables $(\phi,\vartheta)$ while leaving the quasi-velocities $(\omega,v )$ unchanged.} $\mathrm{Ro}:\I^+ \to \I^-$ by
$$ \mathrm{Ro}(z^+):= ([\phi]+\Omega(z^+), [-\vartheta+\pi],\omega,v)  ,$$
where $\Omega:\I^+ \to \mathbb{S}^1$ denotes the continuous $SO(2)$-invariant function defined in Eq. \eqref{eq:def-rot-map}. We also define the \textbf{time-of-flight function} $\Delta t:\I^+\to (0,\infty)$ by Eq. \eqref{eq:fin-Del-t}. We are now ready to state our first main result.
\begin{thm}\label{the:rot-map}
If $z_i^+\in \I^+$ is the $i$-th post-impact state, then the next pre-impact state is determined by the relation 
    $$z^-_{i+1} = \mathrm{Ro}(z^+_i).$$
    Moreover, the elapsed time between consecutive impacts is 
    $$t_{i+1}-t_i = \Delta t(z_i^+).$$
\end{thm}
We next consider the impact map $\mathrm{I}:\mathcal{I}^- \to \mathcal{I}^+$, defined by Eq. \eqref{eq:imp-map-sle}, and define the map
$$ \T\coloneqq \mathrm{I}\circ\mathrm{Ro}:\I^+\to \I^+. $$
The first consequence of Theorem \ref{the:rot-map} is the following corollary.
\begin{corollary}
The map $\T:\I^+\to \I^+$ completely determines the dynamics of the knife-edge billiard in the disk via the recurrence relation
    $$\T(z^+_i) = z^+_{i+1}.$$   
\end{corollary}

\begin{figure}
    \centering
    \includegraphics[width=0.28\linewidth]{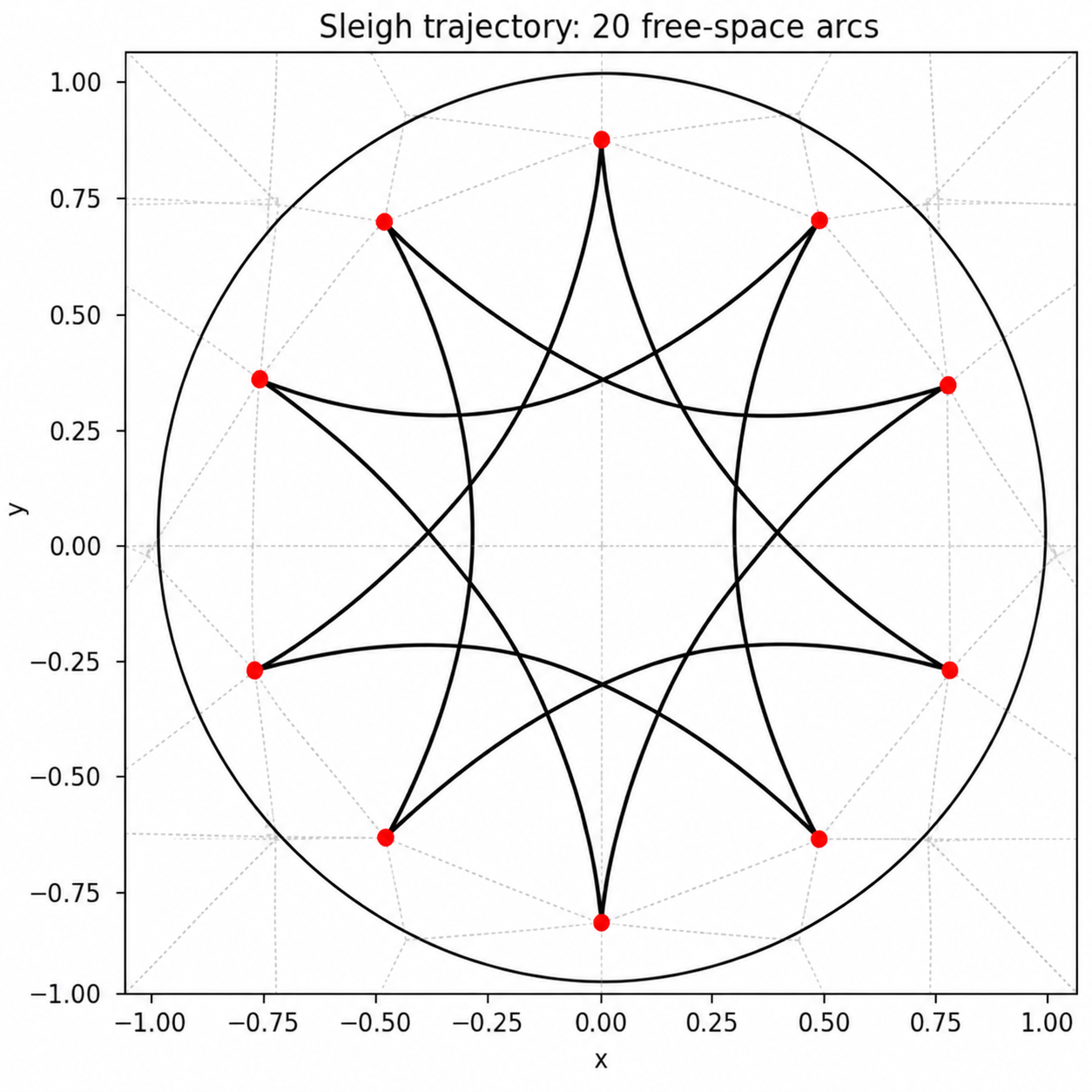}\quad \includegraphics[width=0.28\linewidth]{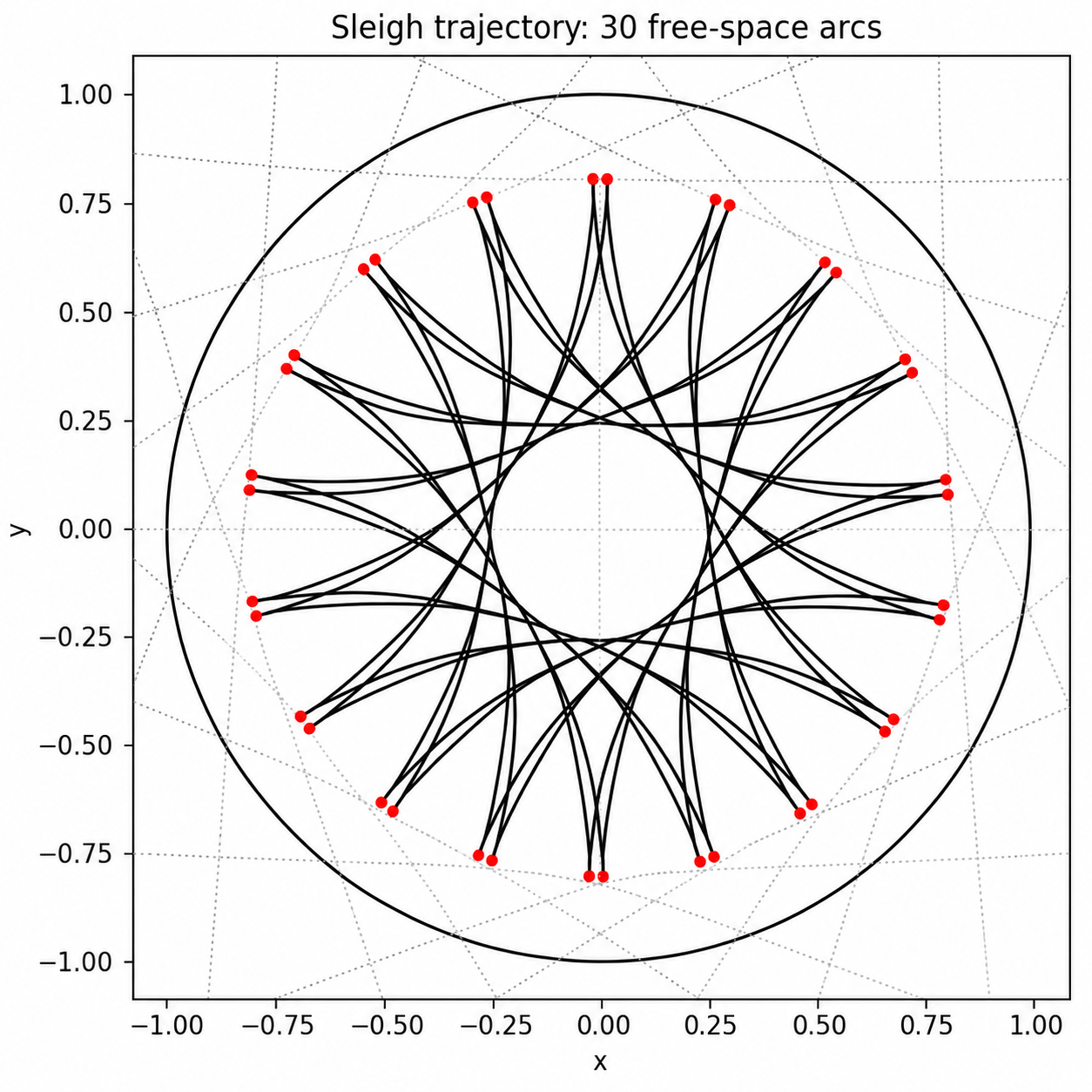}
    \caption{Examples of a periodic and a dense orbit for perpendicular impacts. In this case, the circles $\mathcal{C}^0(z_0^+)$ and $\mathcal{C}^1(z_0^+)$ overlap.}
    \label{fig:per-den-mot}
\end{figure}

Our second main result is the following.
\begin{thm}\label{mainthe:slei}
    For every $i \in \mathbb{Z}$, the iterates of $~\T^2$ are given by 
    $$ \T^{2i}(z) = (\phi+i\varphi(z),\vartheta,\omega,v).$$
    Here, $\varphi: \I^+\to \mathbb{S}^1$ denotes the continuous $SO(2)$-invariant function defined as
     \begin{equation}\label{eq:def-varphi-fuc}
         \varphi(z^+) := \Omega(z^+) + \Omega(T(z^+)).
     \end{equation} 
\end{thm}

Thus, Theorem \ref{mainthe:slei} shows that $\T^2$ is a \textbf{rotation map}.\footnote{Here, we use the convention that a rotation map shifts points uniformly, whereas a circle map need not.} This generalizes the classical billiard in a disk, whose dynamics are governed by a rotation map \cite[Chapter 2]{tabachnikov2005geometry}. Although $\T$ itself is not a rotation map, $\T^2$ is.   As in the classical case, this integrability is linked to caustics \cite[Chapter 5]{tabachnikov2005geometry}, which we generalize as follows:


For simplicity, assume that the billiard table $\mathbb{D}$ is the disk of radius $R$ centered at the origin. Then, for each $z^+\in \I^+$, define the curve
$$ \mathcal{C}(z^+) \coloneqq  \mathcal{C}^0(z^+) \cup \mathcal{C}^1(z^+) \subset \mathbb{D}, $$
where $\mathcal{C}^0(z^+)$ and $\mathcal{C}^1(z^+)$ are the circles defined by
\begin{equation*}
    \begin{split}
        \mathcal{C}^0(z^+) & \coloneqq \big\{ (x,y) \in \R^2: x^2+y^2 = \mathrm{f}^2(z^+) \big\}, \\
        \mathcal{C}^1(z^+) & \coloneqq  \big\{ (x,y) \in \R^2: x^2+y^2 = \mathrm{f}^2\big(\T(z^+)\big) \big\}. \\
    \end{split}
\end{equation*}
Here $\mathrm{f}:\I^+\to [0,R)$ is the continuous function defined in Eq. \eqref{eq:func-f-def}. 
We can now state our third main theorem.

\begin{thm}\label{thm:2}
    The knife-edge billiard in the disk of radius $R$ admits a family of generalized caustics 
    $$\mathcal{F} = \{\mathcal{C}(z^+) \}_{z^+\in \I^+}. $$
    More precisely, if $z_0^+\in \I^+$, then,  for every $i\in \mathbb{Z}$, the following statements hold:
    \begin{itemize}
        \item The trace of the midpoint $p(t)\in \mathbb{D}$ under the smooth dynamics with initial state $z^+_{2i}\in \I^+$ is tangent to $\mathcal{C}^0(z_0^+)$. 
        \item The trace of the midpoint $p(t)\in \mathbb{D}$ under the smooth dynamics with initial state $z^+_{2i+1}\in \I^+$ is tangent to $\mathcal{C}^1(z_0^+)$. 
    \end{itemize}
    Moreover, the midpoint corresponding to the state $z^+_{i}\in \I^+$ lies on a circle for all $i\in \mathbb{Z}$.
    \end{thm}

The following corollary is a direct consequence of Theorems \ref{mainthe:slei} and \ref{thm:2}.
\begin{corollary}
The knife-edge billiard in the disk has the following properties. 
    \begin{itemize}
        \item If $\varphi(z_0)$ is a rational multiple of $2\pi$, say $\varphi(z_0) = \frac{2\pi a}{b}$ with $a,b \in\mathbb{Z}$, $b>0$, and $gcd(a,b)= 1$, then every midpoint orbit is periodic. In this case, the orbit is $b$-periodic and completes $|a|$ turns around the disk. 
        \item  If $\varphi(z_0)$ is not a rational multiple of $2\pi$, then the midpoint orbit is dense in the annulus bounded by the caustic and the circle traced by the knife-edge's midpoint at the impact times.
    \end{itemize}
\end{corollary}
These consequences are illustrated in Figs. \ref{fig:genric-mot}, \ref{fig:per-den-mot}, and \ref{fig:v-zero}.

\begin{figure}
    \centering
    \includegraphics[width=0.28\linewidth]{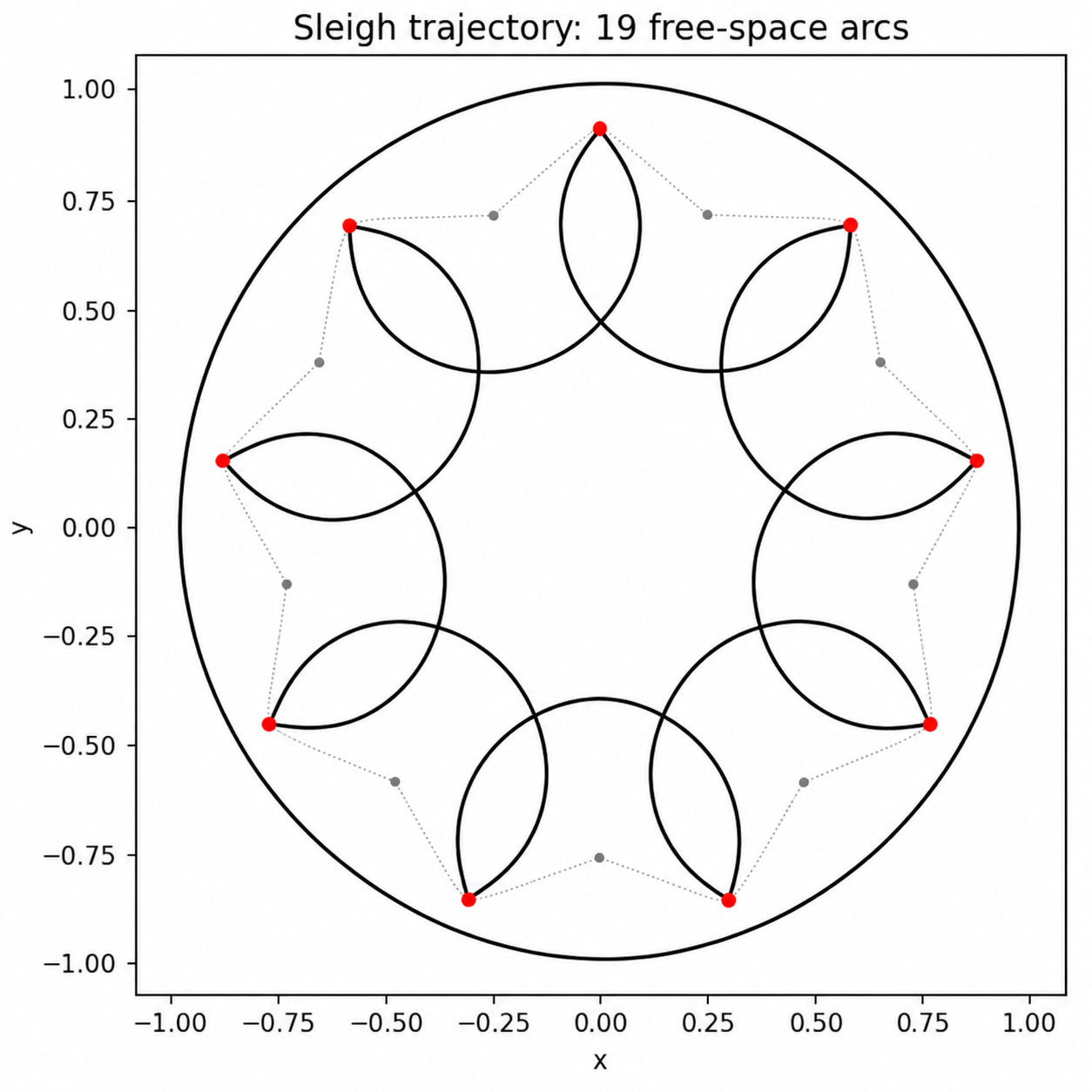}\quad \includegraphics[width=0.28\linewidth]{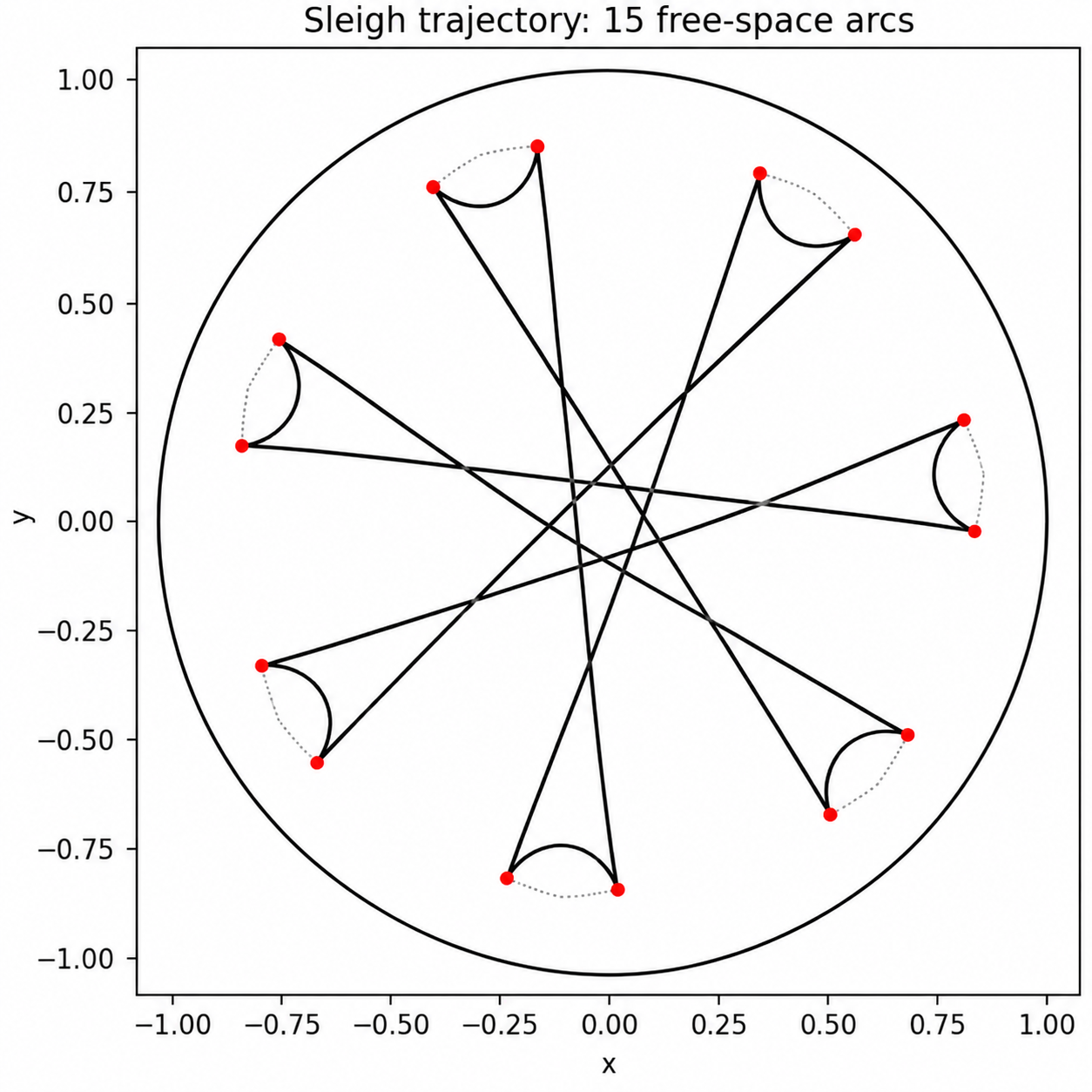}
    \caption{Examples of periodic orbits for the singular dynamics. The left panel shows the case $\omega \neq 0$, $v =0$, and the right panel shows the case $\omega =0$, $v\neq 0$.}
    \label{fig:v-zero}
\end{figure}


\subsection*{Code availability} The second author developed a code implementing the map $\T$ to illustrate the dynamics of the knife-edge billiard in the disk.\footnote{\href{https://github.com/DanielHou315/Knife-Edge-Billard-In-Disk}
     {GitHub repository}} Figs. \ref{fig:genric-mot}, \ref{fig:per-den-mot}, and \ref{fig:v-zero} were generated using this code. Since the code played a fundamental role in developing our intuition about the problem, we make it publicly available.

\subsection*{Acknowledgments}
We are grateful to the Department of Mathematics at the University of Michigan for fostering the collaboration that led to this paper. In particular, we thank Nir Gadish, director of the LoG(M) project, for coordinating the research teams through which A.B.D. and H.H. met. A.B.D. also thanks his Ph.D. advisor, Richard Montgomery, for suggesting the use of Euclidean geometry to approach the project and encouraging him to pursue this line of research. A.M.B. and W.C. were supported in part by AFOSR grant FA9550-23-1-0400 (MURI), and A.M.B. by AFOSR grant FA9550-32-1-0215 and  NSF grants  DMS-2103026 and DMS-2605296.

\subsection*{Structure of the paper}

Section \ref{sec:set-up-pro} formulates the knife-edge billiard in the disk. Subsection \ref{ap:hyb-sys} reviews the basic concepts of hybrid systems and presents the formula for the nonholonomic elastic impact map in Proposition \ref{prp:pla-imp-}. Subsection \ref{sec:chap-sle} introduces the knife-edge billiard and equips the impact space with coordinates. Section \ref{sec:the-2} proves Theorem \ref{the:rot-map} by analyzing the different dynamical regimes, distinguishing between the case $\omega \neq 0$ and the case $\omega  =0$. Finally, Section \ref{sec:proof-main-the} computes the impact map $\mathrm{I}:\I^-\to \I^+$ in the \(z\)-coordinates and proves Theorems \ref{mainthe:slei} and \ref{thm:2}.

\section{Setting Up the Problem}\label{sec:set-up-pro}

This section formulates the knife-edge billiard in the disk. Subsection \ref{ap:hyb-sys} reviews the basic definitions of a nonholonomic hybrid system and the associated impact space. Subsection \ref{subsec:ela-non-hol-imp} presents the nonholonomic elastic impact law in terms of the horizontal gradient. Subsection \ref{subsec:qua-vel} introduces moving frames and quasi-velocities, and rewrites the impact law as a reflection on the space of quasi-velocities. Subsection \ref{sec:chap-sle} introduces the knife-edge as a nonholonomic system. Finally, Subsection \ref{subsec:set-bil} equips the space of impact states with coordinates and introduces the $SO(2)$-invariant coordinate $\vartheta$ that reflects the underlying symmetry of the system.

\subsection{Hybrid Systems}\label{ap:hyb-sys}

We briefly summarize the framework presented in \cite{clark2019bouncing}. Roughly speaking, a hybrid system is a mathematical model of a phenomenon that exhibits both continuous and discrete dynamics \cite{teel2012hybrid}. As we mentioned earlier, a nonholonomic hybrid system is given by the quadruple $(Q,\mathcal{L},\D,\mathcal{S})$. We now recall the definitions of these elements and state the assumptions used in our setting.
\begin{enumerate}
    \item The \textbf{configuration space} $Q$ is a smooth manifold representing the possible configurations of the system. We assume that $Q$ is endowed with a Riemannian metric $g(\cdot,\cdot):TQ\times TQ \to \R$, which defines the kinetic energy of the system as $\mathrm{K}(q,\dot{q}) \coloneqq  \frac{1}{2}g_q(\dot{q},\dot{q})$.

    \item The \textbf{Lagrangian function} $\La:TQ\to \R$ is a smooth function that determines the smooth dynamics through the Lagrange-d’Alembert principle. We assume that the Lagrangian $\La$ is \textbf{natural}; that is, it has the form kinetic energy minus potential energy:
    $$\La(q,\dot{q}) = \mathrm{K}(q,\dot{q}) - \mathrm{V}(q).$$

    \item The \textbf{constraint distribution} $\D \subset TQ$ is a nonintegrable distribution that assigns to each point $q\in Q$ a linear subspace $\D_q \subset T_qQ$ of admissible velocities, while imposing no constraints on the admissible configurations. We assume that $\D$ has constant rank, meaning that $\dim \D_q$ is independent of $q$. Furthermore, $\D$ is defined as the kernel of a collection of linearly independent one-forms:
    $$ \D_q = \bigcap_{k=1}^m \ker \;\Theta_q^k.$$

    \item The \textbf{impact surface} $\Si\subset Q$ is a smooth embedded codimension-1 submanifold that represents the physical boundary at which smooth motion ceases and impacts occur. We assume that the impact surface $\mathcal{S}$ is defined as the level set of a smooth function $h:Q \to \R$.
\end{enumerate}

Recall that a pair $(q,\dot{q})\in TQ$ is called a \textbf{state}. In this context, pre-impact and post-impact refer to the states of the system immediately before and immediately after a collision with the impact surface $\Si$, respectively. More precisely, let $q(t) \in Q$ denote the smooth motion on the time interval $(t_i,t_{i+1})$:
\begin{itemize}
    \item A state $(q,\dot{q}^-)$ is called \textbf{pre-impact} if 
 $$\lim_{t\to t_{i+1}^-}(q(t),\dot{q}(t)) = (q,\dot{q}^-)\in \D|_\Si. $$
 \item A state $(q,\dot{q}^+)$ is called \textbf{post-impact} if 
 $$\lim_{t\to t_i^+}(q(t),\dot{q}(t)) = (q,\dot{q}^+)\in \D|_\Si. $$
\end{itemize}

The standard formulation of Birkhoff billiards considers only regular impacts \cite[Section 1]{ALBERS2018822}. In contrast, as we see in the knife-edge billiard, nonholonomic billiards also admit tangential impacts \cite{KRYZHEVICH2024128018}. We therefore introduce the corresponding definitions. Let $dh_q:T_qQ \to \R$ denote the differential of the function $h$ at $q$, where $h$ is a defining function for the impact surface $\Si$.

\begin{itemize}
    \item An impact state $(q,\dot{q})$ is called \textbf{regular} if 
 $$ dh_q(\dot{q}) \neq 0. $$
 \item An impact state $(q,\dot{q})$ is called \textbf{tangential} if 
 $$ dh_q(\dot{q}) = 0.  $$
\end{itemize}

\subsubsection{Elastic Nonholonomic Impact Map}\label{subsec:ela-non-hol-imp}

For the classical theory of elastic impacts, see \cite{Redner_2004}. In Subsection \ref{ap:hyb-sys}, we assume that the Lagrangian function $\La:TQ\to\R$ is natural. Under this assumption, the fiber derivative (Legendre transform) $F\La:TQ\to T^*Q$ is a global diffeomorphism, defined by $v_q \to D\La_q(v_q) \in T^*_qQ$. Consequently, the Hamiltonian (energy) of the system is given by
$$ H(q,p) = p(v) - \La(q,v), \;\;\text{where}\;\; p = F\La_q(v).$$
With these definitions in place, elastic impacts are characterized by the variational impact equations
\begin{equation}\label{eq:imp}
    \begin{split}
        (F\La^+-F\La^-)\delta q & = 0, \\
        (H^+-H^-)\delta t & = 0. 
    \end{split}
\end{equation}
The impact is called \textbf{free} if $\delta t \neq 0$. As the second equation in Eq. \eqref{eq:imp} shows, a free impact preserves energy.

 Recall that, given a function $h:Q\to \R$ on a Riemannian manifold $(Q,g)$ and a smooth distribution $\D\subset TQ$, the \textbf{horizontal gradient} of $h$ is the unique vector $\horg h \in \D$ such that  
$$ g(\horg  h,X) = dh(X),\;\text{for all}\;X\in \D. $$
See \cite{montgomery2002tour} for details about the horizontal gradient.

The following proposition gives the nonholonomic elastic impact map.

\begin{proposition}\label{prp:pla-imp-}
    If $\La$ is a natural Lagrangian and the one-forms $\{\Theta^k\}_{k=1}^m$ and $dh$ are linearly independent, then the nonholonomic elastic impact map $\mathrm{I}:\D|_{\Si}\to \D|_{\Si}$ is given by $\mathrm{I}(q,\dot{q}^-) = (q,\dot{q}^+)$, where
    \begin{equation*}
       \dot{q}^+ = \dot{q}^- -\frac{2dh(\dot{q}^-)}{g\big(\horg  h, \horg  h\big)} \horg h.
    \end{equation*}
\end{proposition}
In his Ph.D. thesis, the third author, W.C., obtained a similar formula using the orthogonal projection rather than the horizontal gradient \cite[Eq. 6.1.2]{Clark2020Thesis}. The formula in Proposition \ref{prp:pla-imp-} generalizes the classical notion of Birkhoff billiards; see \cite[Eq. 5]{clark2019bouncing}. In this setting, an elastic impact leaves the tangential component of the velocity unchanged and reverses its normal component. Thus, the classical impact map acts as a reflection with respect to the tangent plane to the impact surface \cite[Section 1]{ALBERS2018822}. In the nonholonomic setting, the impact map acts fiberwise as a reflection of $\D$ across $\D \cap T\mathcal{S}$. Moreover, Proposition \ref{prp:pla-imp-} and the definition of the horizontal gradient imply that $\mathrm{I}$  reverses the sign of the normal velocity component, and that the space of tangential impacts is $I$-invariant.\footnote{ Since we assume that the Lagrangian is natural, the nonholonomic equations are time-reversible; that is, if $(q(t),\dot{q}(t))$ is a solution of the nonholonomic equations, then $(q(-t),-\dot{q}(-t))$ is also a solution. Therefore, a set that is forward-invariant is also backward-invariant, and we simply call it invariant.}

Since the proof of Proposition \ref{prp:pla-imp-} involves techniques that are not relevant to the remainder of the paper, we postpone it to Appendix \ref{apd:imap-map}.

\subsubsection{Moving Frame and Quasi-velocities}\label{subsec:qua-vel}

A moving frame is a natural generalization of a coordinate system that facilitates the formulation of nonholonomic equations by exploiting the symmetries of the system \cite{grabowski2009nonholonomic}. In our setting, it also facilitates the computation of the horizontal gradient. 

Let $(Q,\mathcal{L},\D)$ be a nonholonomic Lagrangian system.  The Riemannian metric $g$ induces an orthogonal decomposition $TQ = \D\oplus\D^\perp$. Let $\{X_i\}_{i=1}^n$ and $\{Y_j\}_{j=1}^{m}$ be moving frames for $\D$ and $\D^\perp$, respectively. Every tangent vector $\dot{q} \in TQ$ can be uniquely expressed as
$$ \dot q = v_iX^i+u_\ell Y^\ell,\;\text{where}\;v_i\;\text{and}\;u_\ell\;\text{are called the}\; \textbf{quasi-velocities}. $$
The condition $\dot{q}\in \D$ is equivalent to $u_1=\dots = u_{m} = 0$.

Let $\{X_i^*\}_{i=1}^n\cup \{Y_\ell^*\}_{\ell=1}^{m}$ denote the dual coframe of $\{X_i\}_{i=1}^n\cup \{Y_\ell\}_{\ell=1}^{m}$.  The differential of $h$ is given by
$$ dh = \sum_{i=1}^n X_i(h) X^*_i + \sum_{\ell=1}^m Y_\ell(h) Y^*_\ell. $$
Furthermore, if $\{X_i\}_{i=1}^n$ is an orthonormal moving frame, then the horizontal gradient is given by
\begin{equation}\label{eq:hor-gra}
   \horg h =  \sum_{i=1}^n X_i(h) X_i.  
\end{equation}
Identifying $\dot{q} \in \D$ with $(v_1,\dots,v_n) \in \R^n$, the impact map takes the following form:
\begin{equation*}
    \mathrm{I}(q,v^-_1,\dots,v^-_n) = (q, \Rm(v^-_1,\dots,v^-_n)),
\end{equation*}
where $\Rm:\R^n\to\R^n$ denotes the reflection defined by 
\begin{equation}\label{eq:imp-map-qua}
v^+_i \coloneqq  v_i^--2\sum_{j=1}^n\frac{ v_j^- X_j(h) }{g(\horg h,\horg h)}X_i(h),\qquad \text{for all} \;i=1,\dots,n.
\end{equation}
More precisely, $\Rm:\R^n\to\R^n$ is an orthogonal transformation with respect to the quadratic form 
$$\mathrm{K}(q,\dot{q}), \;\;\text{where}\;\;\dot{q} = \sum_{i=1}^n v_i X_i.$$
This is consistent with elastic impacts conserving energy.

\subsection{The knife-edge}\label{sec:chap-sle}
Here, we model the \textbf{knife-edge} as a rigid line segment of length $2\ell$. Its configuration space is the special Euclidean group $SE(2)$, with local coordinates $q\coloneqq (\theta,p)$, where $[\theta] \in \mathbb{S}^1$ is the orientation\footnote{ Here, we use the convention that the elements of $\mathbb{S}^1$ are called oriented angles, and an unoriented angle takes values in $[0,\pi]$.} angle measured counterclockwise from the positive x-axis and $p\coloneqq (x,y)\in\R^2$ denotes the \textbf{midpoint}. 

The knife-edge imposes a nonholonomic constraint that prohibits motion perpendicular to the blade. Thus,  the constraint distribution is
\begin{equation}\label{eq:cost-dis}
    \D : =\big\{ \dot{q} \in TSE(2):-\dot{x}\sin\theta+\dot{y}\cos\theta = 0 \big\}.
\end{equation}
The standard left-invariant contact distribution $\D$ is spanned by the vector fields
\begin{equation}\label{eq:lef-in-SE2}
    X_1 = \frac{\partial}{\partial \theta} ,\;\;\text{and}\;\; X_2 =  \cos \theta \frac{\partial}{\partial x} + \sin \theta \frac{\partial}{\partial y}.
\end{equation}

Throughout the paper, we consider the force-free motion of the knife-edge with mass $m$, moment of inertia $J$ about its center of mass, and center of mass at the midpoint $p$. Under these assumptions, the Lagrangian function $\La:TSE(2) \to \R$ is purely kinetic and is given by  
\begin{equation}\label{eq:lag-func}
\begin{split}
    \La(q,\dot{q}) & = \frac{1}{2}\big(J\dot{\theta}^2 + m(\dot{x}^2+\dot{y}^2)  \big).\\
\end{split}
\end{equation}
Since the Lagrangian is independent of the configuration $q$, the system is invariant under the action of the Euclidean group \cite{bloch1996nonholonomic,grabowski2009nonholonomic}. This symmetry allows us to reduce the equations of motion. The reduced dynamics are naturally expressed in terms of the quasi-velocities
\begin{equation*}\label{eq:quasi-vel}
    \omega \coloneqq  \dot{\theta}\quad \text{and} \quad v \coloneqq  \cos\theta \dot{x} + \sin\theta \dot{y}
\end{equation*}
associated with the frame $\{X^1,X^2\}$. The corresponding linear identification between a pair $(\omega,v )\in \R^2$ and a constraint velocity $\dot{q} \in \D$ is
\begin{equation}\label{eq:lin-bij}
    (\omega,v ) \to \dot{q} = \omega X_1+v X_2. 
\end{equation}
Henceforth, we represent a state $(q,\dot{q})$ by the triple $(q,\omega,v)$. In these coordinates, the kinetic energy is given by the quadratic form
\begin{equation*}
    \mathrm{K}(\omega,v ) = \frac{1}{2}\big( J \omega^2+mv^2\big). 
\end{equation*}
Because the center of mass of the knife-edge coincides with the contact point, the quasi-velocity dynamics are trivial \cite[Subsection 1.6]{tony}.\footnote{A more complex nonholonomic system is the Chaplygin sleigh, in which the center of mass is displaced from the contact point \cite{borisov2009dynamics}.} We classify the smooth dynamics as \textbf{generic} when $\omega v \neq 0$ and \textbf{singular} when $\omega v = 0$.

\subsubsection{The Billiard System}\label{subsec:set-bil}

We restrict our attention to the knife-edge billiard in the disk of radius $R$. Recall that the disk is assumed to be centered at the origin. Thus, the \textbf{boundary of the billiard table} is
\begin{equation*}\label{eq:round table}
    \mathbb{S}^1(R) \coloneqq  \{(x,y) \in \R^2: \|(x,y)\|^2 = R^2 \}.
\end{equation*}
The boundary $\mathbb{S}^1(R)$ is parametrized by an angle $[\phi] \in \mathbb{S}^1$,  measured counterclockwise from the positive $x$-axis, so that each point $s \in\mathbb{S}^1(R)$ has the form
$$ s =  R(\cos\phi, \sin\phi).$$

Since the knife-edge has no intrinsic front-to-back orientation, the configurations $(\theta,p )$ and $(\theta +\pi,p )$ are indistinguishable. To describe impact states, we therefore adopt the following coordinate convention:
        \begin{equation}\label{eq:def-front-back}
         \begin{split}
         p_f &\coloneqq  p + \ell (\cos\theta, \sin\theta), \\
          p_b &\coloneqq  p - \ell (\cos\theta, \sin\theta). 
        \end{split}
        \end{equation}
We refer to $p_f$ and $p_b$ as the knife-edge's \textbf{front point} and \textbf{back point}, respectively. With this convention, we define the impact surface. We first regard the expressions in  Eq. \eqref{eq:def-front-back} as defining functions $p_f,p_b:SE(2) \to \R^2$. We then define the function $h:SE(2) \to \R$ as
\begin{equation}\label{eq:round table-1}
    h(q)\coloneqq \max\big\{ \|p_f(q)\|^2, \|p_b(q)\|^2 \big\}. 
\end{equation}
Therefore, the \textbf{impact surface} is given by 
$$\Si:=\{ q \in SE(2): h(q) = R^2\}.$$

With this convention in place, we say that a state $(q,\dot{q})$ is an impact state if the following conditions hold:
\begin{itemize}
    \item  Either the front point or the back point lies on the impact surface. 
    \item Apart from the endpoint lying on the boundary, the knife-edge is strictly contained in the disk.\footnote{The only nontrivial motion for which both endpoints lie on the boundary of the table occurs when both points undergo circular motion centered at the origin with radius $R$; in other words, both points remain on the boundary for all time. Since these states do not produce successive impacts with the boundary, we exclude them from our analysis.}
\end{itemize}
An impact state $(q,\dot{q})$ is of \textbf{front-type} if $p_f(q) \in \mathbb{S}^1(R)$ and of \textbf{back-type} if $p_b(q) \in \mathbb{S}^1(R)$. By the second condition above, there are no impact states $(q,\dot{q})$ that are simultaneously front-type and back-type; it follows that either $R^2= \|p_f(q)\|^2$ or $R^2 = \|p_b(q)\|^2$. Consequently, $h$ is smooth in a neighborhood of the space of impact states. 

We now endow the impact space with coordinates; we begin with front-type states. By Eq. \eqref{eq:def-front-back}, an impact state is front-type if there exists a point $s \in \mathbb{S}^1(R)$ such that 
\begin{equation}\label{eq:def-front-type}
  s= p_f. 
\end{equation}
The back point of the knife-edge lies strictly inside the disk if and only if
$$ \|s - 2 \ell (\cos\theta, \sin\theta)\| <R.  $$
Using the expression $s=R(\cos\phi,\sin\phi)$, we see that the above inequality is equivalent to 
\begin{equation}\label{eq:def-front-type-2}
    \begin{split}
      \frac{\ell}{R} &< \cos\big( \phi-\theta\big) . \\
    \end{split}
\end{equation}
We continue with back-type states. Similarly, by Eq. \eqref{eq:def-front-back}, an impact state is back-type if there exists a point $s_b \in \mathbb{S}^1(R)$ such that 
\begin{equation}\label{eq:def-back-type}
  s_b= p_b. 
\end{equation}
The front point of the knife-edge lies strictly inside the disk if and only if
\begin{equation}\label{eq:def-back-type-2}
    \begin{split}
      \frac{\ell}{R} &< - \cos\big( \phi-\theta\big) . \\
    \end{split}
\end{equation}
Thus, each quadruple $([\phi],[\theta],\omega,v)\in \mathbb{T}^2\times(\R^2\setminus\{(0,0)\})$\footnote{The point $(\omega,v ) = (0,0)$ is excluded since its dynamics are trivial.} satisfying either Eqs. \eqref{eq:def-front-type} and \eqref{eq:def-front-type-2} or Eqs. \eqref{eq:def-back-type} and \eqref{eq:def-back-type-2} determines the midpoint $p$ of an impact state of front-type or back-type, respectively. Consequently, we endow the space of impact states with coordinates $([\phi],[\theta],\omega,v)$. 

Although the definition of the impact space is natural in the coordinates $([\phi],[\theta],\omega,v)$, these coordinates do not make the symmetries of the system readily apparent. As noted earlier, the knife-edge system is invariant under rigid motions. This rotational symmetry is preserved in the knife-edge billiard in the disk, since the impact surface is itself invariant under rotations. Specifically, a rotation of the plane $\R^2$ by an angle $\tau \in \mathbb{S}^1$ about the origin induces the following transformation of the coordinates:
$$  ([\phi],[\theta],\omega,\allowbreak v) \to ([\phi]+\tau,[\theta]+\tau,\omega,\allowbreak v). $$
We conclude that Eqs. \eqref{eq:def-front-type-2} and \eqref{eq:def-back-type-2} are invariant under rotations. This observation motivates the following change of coordinates:
\begin{equation}\label{eq:cha-coor}
    ([\phi],[\theta],\omega, v) \to ([\phi],[\vartheta],\omega, v)\coloneqq  ([\phi],[\phi-\theta],\omega, v). 
\end{equation}
To simplify notation, throughout the paper we write $z: = ([\phi], [\vartheta], \omega, v)$.

Thus, via Eqs. \eqref{eq:def-front-type} and \eqref{eq:cha-coor}, we identify the space of front-type impact states with the coordinate set
\begin{equation}\label{eq:def-fro-po}
    \begin{split}
        \I_f\coloneqq  \Big\{ z\in \mathbb{T}^2\times(\R^2\setminus\{(0,0)\}) : \frac{\ell}{R} &< \cos\vartheta  \Big\} . \\
    \end{split}
\end{equation}
By an analogous argument,  using Eqs. \eqref{eq:def-back-type} and \eqref{eq:cha-coor}, we identify the space of back-type impact states  with the coordinate set
\begin{equation}\label{eq:def-bac-po}
    \begin{split}
        \I_b\coloneqq  \Big\{ z \in \mathbb{T}^2\times(\R^2\setminus\{(0,0)\}) : \frac{\ell}{R} &< -\cos\vartheta\Big\} . \\
    \end{split}
\end{equation}
The space of impact states is then
$$  \I \coloneqq  \I_f \cup \I_b. $$

We identify the coordinate $z$ with the corresponding impact state $(q(z),\dot{q}(z))$. More precisely, $q(z) = (\theta(z),p(z))$, where $\theta(z)=\phi-\vartheta$ and the midpoint $p(z)$ is determined by Eq. \eqref{eq:def-front-type} when $z$ is front-type and by Eq. \eqref{eq:def-back-type} when $z$ is back-type. The velocity  $\dot{q}(z)$ is given by Eq. \eqref{eq:lin-bij}. Throughout the paper, we use this convention to specify the initial states for the smooth dynamics following an impact.  To this end, we proceed to define pre-impact states and post-impact states. We begin with the front-type case. If $z\in \I_f$, then expressing $h$ in the $z$-coordinates  gives
\begin{equation}\label{eq:h-dif-fro}
 \frac{1}{ R}dh_{q(z)}(\dot{q}(z)) =  \ell\omega \sin\vartheta + v\cos\vartheta.  
\end{equation}
Hence, the spaces of regular front-type pre-impact and post-impact states are
\begin{equation}\label{eq:def-reg-fro-im}
    \begin{split}
        \mathcal{RI}_f^- & \coloneqq  \Big\{ z \in \I_f : 0 < \ell\omega \sin\vartheta + v\cos\vartheta\Big\},\\
        \mathcal{RI}_f^+ & \coloneqq  \Big\{ z \in \I_f : 0 > \ell\omega \sin\vartheta + v\cos\vartheta\Big\}.\\
    \end{split}
\end{equation}
Finally, if $z\in \I_b$, we obtain
\begin{equation}\label{eq:h-dif-back}
\frac{1}{ R} dh_{q(z)}(\dot{q}(z)) =  \ell\omega \sin\vartheta - v\cos\vartheta .    
\end{equation}
Thus, the spaces of regular pre-impact and post-impact states of back-type are defined by
\begin{equation}\label{eq:def-reg-bac-im}
    \begin{split}
        \mathcal{RI}_b^- & \coloneqq  \Big\{ z\in \I_b : 0 <\ell\omega \sin\vartheta - v\cos\vartheta\Big\},\\
        \mathcal{RI}_b^+ & \coloneqq  \Big\{ z\in \I_b : 0 > \ell\omega \sin\vartheta - v\cos\vartheta\Big\}.\\
    \end{split}
\end{equation}
Therefore, the spaces of regular pre-impact and post-impact states are defined by
$$ \mathcal{RI}^-:= \mathcal{RI}^-_f\cup \mathcal{RI}^-_b,\quad \mathcal{RI}^+:= \mathcal{RI}^+_f\cup \mathcal{RI}^+_b. $$
The spaces of tangential-impact states of front-type and back-type are defined by
\begin{equation}\label{eq:def-tan}
    \begin{split}
        \mathcal{TI}_f & \coloneqq   \Big\{ z \in \I_f : 0 = \ell\omega \sin\vartheta + v\cos\vartheta\Big\},\\
        \mathcal{TI}_b & \coloneqq \Big\{ z\in \I_b : 0 = \ell\omega \sin\vartheta - v\cos\vartheta\Big\}.
    \end{split}
\end{equation}

The space of tangential-impact states is therefore
\begin{equation*}
    \begin{split}
        \mathcal{TI} \coloneqq  & \mathcal{TI}_f \cup \mathcal{TI}_b .
    \end{split}
\end{equation*}
\begin{remark}\label{rmk:tan-imp}
    The tangential-impact states have the following properties. 
    \begin{itemize}
    \item Eq. \eqref{eq:def-tan} implies that tangential-impact states with $v = 0$ must satisfy $\vartheta = \pm\frac{\pi}{2}$. 
    \item Eqs. \eqref{eq:def-fro-po} and \eqref{eq:def-bac-po} imply that there are no tangential-impact states with $\omega = 0$.
    \end{itemize}  
\end{remark}

Consequently, the spaces of pre-impact and post-impact states are
$$ \I^- = \mathcal{RI}^-  \cup \mathcal{TI},\;\;\text{and}\;\; \I^+ = \mathcal{RI}^+  \cup \mathcal{TI}. $$
Finally, the space of impact states is
$$ \I = \I^-\cup \I^+.$$
Now that we have formally defined the impact states, we present some basic properties of the coupled circle map.
\begin{remark}
Eqs. \eqref{eq:def-reg-fro-im}, \eqref{eq:def-reg-bac-im}, and \eqref{eq:def-tan} imply that   $\mathrm{Ro}:\I^+\to \I^-$ is well defined.    
\end{remark}

Assuming temporarily that Theorem \ref{the:rot-map} holds, we obtain the following corollary, which describes the types of consecutive impacts.
\begin{corollary}\label{def:cor-rot-map}
    Let $z_i^+$ and $z_{i+1}^-$ denote the $i$-th post-impact state and the $(i+1)$-st pre-impact state, respectively.  Then consecutive impacts satisfy the following relations:
    \begin{itemize}
        \item If $z_i^+\in \mathcal{RI}_f^+$, then $z_{i+1}^-\in \mathcal{RI}_b^-$, and if $z_i^+\in \mathcal{RI}_b^+$, then $z_{i+1}^-\in \mathcal{RI}_f^-$.
        \item If $z_i^+\in \mathcal{TI}_f$, then $z_{i+1}^-\in \mathcal{TI}_b$, and if $z_i^+\in \mathcal{TI}_b$, then $z_{i+1}^-\in \mathcal{TI}_f$.
    \end{itemize}
\end{corollary}

\begin{proof}
     Eqs. \eqref{eq:def-reg-fro-im}, \eqref{eq:def-reg-bac-im}, and \eqref{eq:def-tan} imply the following relations:
    \begin{equation*}
    \begin{split}
        \mathrm{Ro}&: \mathcal{RI}_f^+ \to \mathcal{RI}_b^-,\; \;\;\; \mathrm{Ro}: \mathcal{RI}_b^+ \to \mathcal{RI}_f^-,  \\
    \end{split}
    \end{equation*} 
    and
        \begin{equation*}
    \begin{split}
        \mathrm{Ro}&: \mathcal{TI}_f \to \mathcal{TI}_b,\;\;\; \;\;\; \mathrm{Ro}: \mathcal{TI}_b \to \mathcal{TI}_f. \\
    \end{split}
    \end{equation*} 
    This proves the corollary.
\end{proof}
Hence, front-type and back-type impacts necessarily alternate. In other words, a smooth trajectory cannot exhibit two consecutive impacts of the same type.

We conclude this section by noting that the midpoint of the knife-edge moves along a circle when $\omega\neq 0$, and along a straight line when $ \omega = 0$, as discussed in more detail in the next section. Therefore, the knife-edge billiard is well-posed in the sense that the motion remains inside the disk after a tangential impact.


\section{The Coupled Circle Map}\label{sec:the-2}

In this section, we prove Theorem \ref{the:rot-map}. To this end, we first introduce the function $\Omega:\I^+\to \mathbb{S}^1$. We begin by considering the following subset:
\begin{equation}\label{eq:def-sin-im-v-neq-zer}
    \mathcal{RI}_v^+ \coloneqq  \{ z^+\in \mathcal{RI}^+: \omega = 0\}.
\end{equation}

 With this definition in place, $\Omega:\I^+\to \mathbb{S}^1$ is given by
    \begin{equation}\label{eq:def-rot-map}
        \Omega(z^+) \coloneqq  \begin{cases}
            [-2\sgn(\omega)\alpha(z^+)],\;\;&\text{if}\;z^+\in \mathcal{I}^+\setminus \mathcal{RI}_v^+, \\
            [-2\vartheta+\pi],\;\;&\text{if}\;z^+\in \mathcal{RI}_v^+.\\
        \end{cases}
    \end{equation}
    Here, $\sgn:\R\setminus\{0\} \to \{\pm 1\}$ denotes the sign function, and $\alpha:\I^+\setminus \mathcal{RI}_v^+ \to [0,\pi]$ is the continuous $SO(2)$-invariant function introduced in the following subsection.

Our strategy for proving Theorem \ref{the:rot-map} is to establish the identity $\mathrm{Ro}(z^+_i)= z^-_{i+1}$ by proving the corresponding relations for the $\phi$-coordinate and then the $\vartheta$-coordinate. 

To determine the $\phi$-coordinate, we consider the $i$-th post-impact state $z_i^+$ and the $(i+1)$-st pre-impact state $z_{i+1}^-$ and prove
\begin{equation}\label{eq:gen-indet-1}
\begin{split}
         [\phi_{i+1}] & = [\phi_{i}] + \Omega(z_i^+) .
\end{split}
\end{equation}

To determine the $\vartheta$-coordinate, we consider the states $z_i^+$ and $z_{i+1}^-$ and prove that
\begin{equation}\label{eq:gen-indet-2}
\begin{split}
    \vartheta_{i+1} & = -\vartheta_i -\pi(2k+1),\;\text{for some}\;k \in \mathbb{Z}.
\end{split}
\end{equation}
Finally, interpreting Eq. \eqref{eq:gen-indet-2} modulo $2\pi$, we obtain $ [\vartheta_{i+1}] = [-\vartheta_{i}+\pi]$, which, together with Eq. \eqref{eq:gen-indet-1}, yields the desired result.
 
The proofs of the identities in Eqs. \eqref{eq:gen-indet-1} and \eqref{eq:gen-indet-2} rely on elementary Euclidean geometry. To justify the geometric constructions involved, we first establish the relations between front-type and back-type impacts stated in Corollary \ref{def:cor-rot-map} without appealing to Theorem \ref{the:rot-map}. 
Consequently, the proof of Theorem \ref{the:rot-map} is the most technical part of the paper, as it requires a separate analysis of the different dynamical regimes.

The proof is organized by dynamical regime. Subsection \ref{subsec:gen-imp-sta} treats the generic case $\omega v \neq 0$, while Subsections \ref{sub-sec:imp-st-sing-ome-neq-zer} and \ref{subsec:smo-sing-dy} consider the singular cases $\omega \neq 0$, $v=0$ and $\omega =0 $, $v \neq 0$, respectively. 


\subsection{Construction of the Function \texorpdfstring{\(\alpha\)}{} }

Although this construction may initially appear artificial, its natural application becomes evident in Subsections \ref{subsec:gen-imp-sta} and \ref{sub-sec:imp-st-sing-ome-neq-zer}.

\begin{figure}
    \centering
    \includegraphics[width=0.47\linewidth]{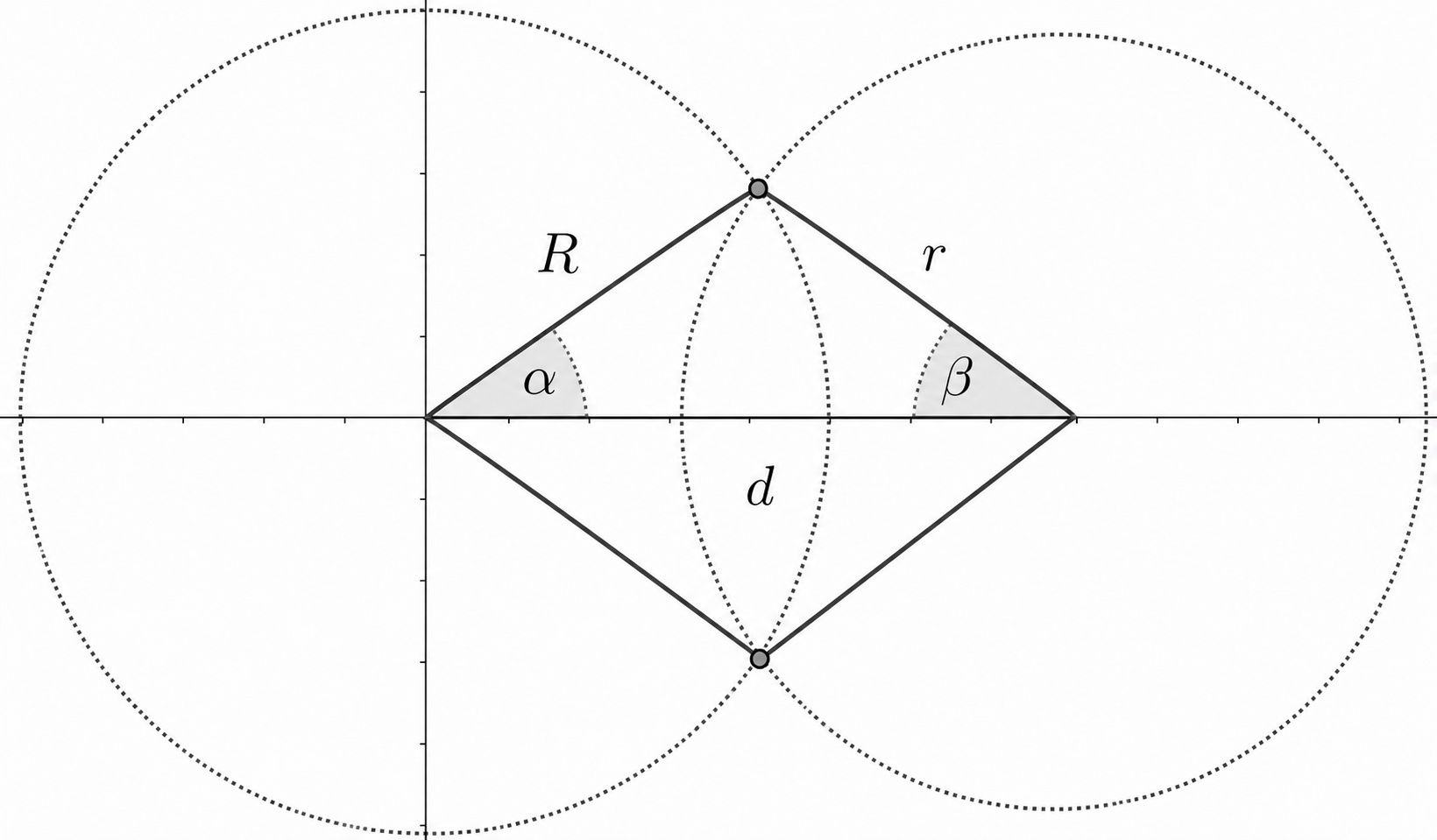}\quad
    \includegraphics[width=0.32\linewidth]{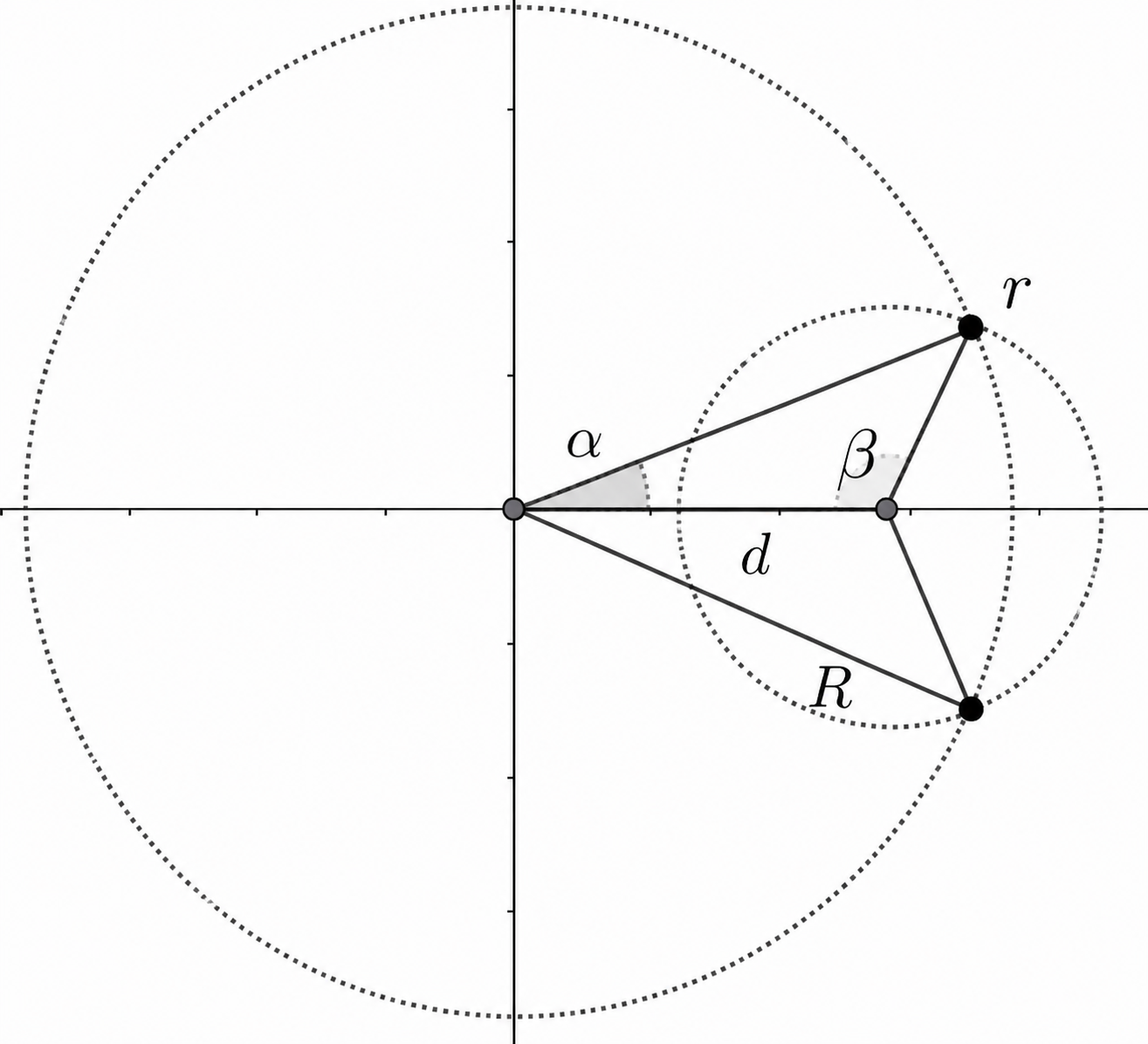}
    \caption{Both images display the circles with radii $R$ and $r$, together with the angles $\alpha$ and $\beta$ that determine their intersection points. The left image shows the case $\beta \in (0,\frac{\pi}{2})$, whereas the right image shows the case $\beta \in (\frac{\pi}{2},\pi)$. }
    \label{fig:circle-1}
\end{figure}

This construction relies on the following classical geometric formulas. Consider two circles of radii $R$ and $r$ whose centers are separated by a distance $d$. The circles intersect in two distinct points, forming a lens-shaped region, if and only if $|R-r| < d<R+r$. The intersection is internally tangent if $|R-r| = d$.  Furthermore, the angles $\alpha$ and $\beta$  shown in Fig. ~\ref{fig:circle-1} are given by
    \begin{equation}\label{eq:def-al-bet}
        \begin{split}
            \alpha  =  \arccos\big( \frac{d^2+R^2-r^2}{2dR}\big),\quad\text{and} \quad\beta  = \arccos\big( \frac{d^2-R^2+r^2}{2dr}\big). \\
        \end{split}
    \end{equation}
    For the internally tangent intersection, we have $\alpha = 0$ and $\beta = \pi$. 
    

    Our goal is to use the above formulas to construct the functions $\alpha,\beta:\I^+\setminus \mathcal{RI}^+_v \to [0,\pi]$; $\beta$ plays a fundamental role in defining the time-of-flight function. In our setting, $R$ is a constant determined by the disk's size, and $d,r:\I^+\setminus \mathcal{RI}^+_v \to [0,\infty)$ are auxiliary functions defined as follows.

  Consider an impact state $z^+\in \I^+\setminus \mathcal{RI}^+_v$. In the next two subsections, we consider the smooth dynamics $(\theta(t),p(t))$ determined by the impact state $z^+$, and we show that the midpoint $p(t)$ traces an arc centered at
  \begin{equation}\label{eq:def-cent}
  \mathrm{c}(z^+) \coloneqq \frac{v}{\omega}(- \sin(\phi-\vartheta),\cos(\phi-\vartheta)) +p(z^+) \in \R^2.  
  \end{equation}
  The first auxiliary function $d:\I^+\setminus \mathcal{RI}^+_v \to [0,\infty)$ is the quantity $d(z^+) \coloneqq \|\mathrm{c}(z^+) \|$. We now derive an explicit expression for $d$. Suppose first that $z^+\in \I_f^+$. Using Eq. \eqref{eq:def-front-back}, we express $\mathrm{c}(z^+)$ as follows:
$$\mathrm{c}(z^+) = R(\cos\phi,\sin\phi) + \ell (\cos(\phi-\vartheta),\sin(\phi-\vartheta)) + \frac{v}{\omega}(-\sin(\phi-\vartheta),\cos(\phi-\vartheta)).$$
We then compute  $\|\mathrm{c}(z^+)\|$ to obtain
\begin{equation*}
  d_f(z^+) \coloneqq \sqrt{R^2+\ell^2+(\frac{v}{\omega})^2+2R\ell \cos(\vartheta)+2\frac{R v}{\omega}\sin\vartheta}.  
\end{equation*}
Analogously, for $z^+\in \I_b$,  we obtain
\begin{equation*}
  d_b(z^+) \coloneqq \sqrt{R^2+\ell^2+(\frac{v}{\omega})^2-2R\ell \cos(\vartheta)+2\frac{R v}{\omega}\sin\vartheta}.  
\end{equation*}
Therefore, we define $d:\I^+\setminus \mathcal{RI}^+_v \to [0,\infty) $
as 
\begin{equation}\label{eq:d-for-fron-type}
    d(z^+) := \begin{cases}
        d_f(z^+) & \text{if} \; z^+ \in \mathcal{I}^+_f \setminus \mathcal{RI}^+_v,\\
        d_b(z^+) & \text{if} \; z^+ \in \mathcal{I}^+_b \setminus \mathcal{RI}^+_v.\\
    \end{cases}
\end{equation}

The second auxiliary function $r:\mathcal{I}^+\setminus\mathcal{RI}^+_v \to [\ell,\infty)$ is defined by
\begin{equation}\label{eq:def-rad-mot}
    r(z^+) \coloneqq  \sqrt{\big(\frac{v}{\omega}\big)^2+\ell^2}. 
\end{equation}

By construction, $d$ and $r$ are continuous and $SO(2)$-invariant since they are independent of $\phi$.


    \begin{definition}\label{def:alp-bet-gam-func}
        The functions $\alpha,\beta: \mathcal{I}^+\setminus\mathcal{RI}^+_v \to [0,\pi]$ are defined by Eqs. \eqref{eq:def-al-bet}. Here, $d:\mathcal{I}^+\setminus\mathcal{RI}^+_v \to [0,\infty)$ and $r:\mathcal{I}^+\setminus\mathcal{RI}^+_v \to [\ell,\infty)$ are auxiliary functions defined by Eqs. \eqref{eq:d-for-fron-type} and \eqref{eq:def-rad-mot}, respectively. We emphasize that $\alpha$ and $\beta$ are unoriented angles and are invariant under rotations. 
    \end{definition}

    The functions $\alpha$ and $\beta$ are smooth on $\mathcal{I}^+\setminus(\mathcal{RI}^+_v\cup \mathcal{TI})$ and  take values in $(0,\pi)$. Moreover, they extend continuously to $\mathcal{TI}$, where $\alpha(z^+) = 0$ and $\beta(z^+) = \pi$ if $z^+ \in \mathcal{TI}$. 
    

\subsection{Generic Case }\label{subsec:gen-imp-sta}

In the generic case, the kinetic energy is the sum of the translational and rotational kinetic energies. Let $z^+\in \mathcal{I}^+$ be a post-impact state with $0 \neq \omega v$. Then the solution to the nonholonomic equations of motion with initial state $z^+$ is given by
\begin{equation*}\label{eq:mot-gen}
\theta(t)  = \omega t+\theta(z^+), \quad p(t) = \mathrm{c}(z^+) + \frac{v}{\omega}\big(\sin\theta(t),-\cos\theta(t)\big).
\end{equation*}
Here $\mathrm{c}(z^+)\in \R^2$ is a constant point defined by Eq. \eqref{eq:def-cent}. It follows that the midpoint $p(t)$ of the knife-edge moves along a circle centered at $\mathrm{c}(z^+) $ with radius $|\frac{v}{\omega}|$.

The points $p_f$ and $p_b$ provide a convenient way to define the impact states. However, when $v\neq 0$, the distinction between them is dynamically irrelevant in the following sense: the transformation
\begin{equation}\label{eq:fron-back-sym}
     (\theta,p,\omega,v) \to (\theta+\pi,p,\omega,-v)
\end{equation}
interchanges the front and back points while leaving the trajectory traced by the midpoint $p(t)$ unchanged. Thus, from a physical point of view, the dynamically relevant distinction is not between the front and back points, but between the leading and trailing points. Accordingly, we define
\begin{equation}\label{eq:lead-point}
    \begin{split}
  p_l(t) \coloneqq  p(t) +\sgn(v) \ell (\cos\theta(t),\sin\theta(t)), \\
  p_t(t) \coloneqq  p(t) - \sgn(v) \ell (\cos\theta(t),\sin\theta(t)). 
    \end{split}
\end{equation}
We refer to $p_l(t)$ and $p_t(t)$ as the \textbf{leading} and \textbf{trailing} points of the knife-edge. Since these points are intrinsically defined by the dynamics and are invariant under the transformation in Eq. \eqref{eq:fron-back-sym}, they each move along a circle centered at $\mathrm{c}(z^+)$ with radius $r(z^+)$, where \(c\) and \(r\) are defined in Eqs. \eqref{eq:def-cent} and \eqref{eq:def-rad-mot}, respectively. 

We denote by $\gamma$ the unoriented angle $\angle p(t)\mathrm{c}(z^+)p_l(t)$; see Fig. \ref{fig:leading-point}. More precisely, the function $\gamma: \mathcal{I}^+\setminus\mathcal{RI}^+_v \to [0,\frac{\pi}{2}]$ is given by
\begin{equation}\label{eq:gam-def}
    \gamma(z^+) \coloneqq  \arcsin\Big(\frac{\ell}{r}\Big). 
\end{equation}
The function $\gamma$ is well defined since $0<\ell\leq r$, and it is $SO(2)$-invariant. By symmetry of the knife-edge, the angle $\gamma(z^+)$ is also $\angle p_t(t)\mathrm{c}(z^+)p(t)$.

Before continuing, we record two observations about the generic smooth dynamics that will be used throughout this section.

\begin{remark}\label{rem:omega-sign} The generic smooth dynamics have the following properties. 
    \begin{itemize}
        \item To motivate the terminology \enquote{leading point} and \enquote{trailing point,} observe that the following identity holds; see Fig. \ref{fig:leading-point}:  
        $$p_l(t) = p_t(t+\frac{2\gamma}{|\omega|}) .$$
        \item The circular motion of the knife-edge is clockwise when $~\omega<0$ and counterclockwise when $~\omega>0$. This dichotomy plays a fundamental role, as will become apparent in the proof of Eq. \eqref{eq:def-rot-map}. 
    \end{itemize}
\end{remark}

\begin{figure}
    \centering
    \includegraphics[width=0.39\linewidth]{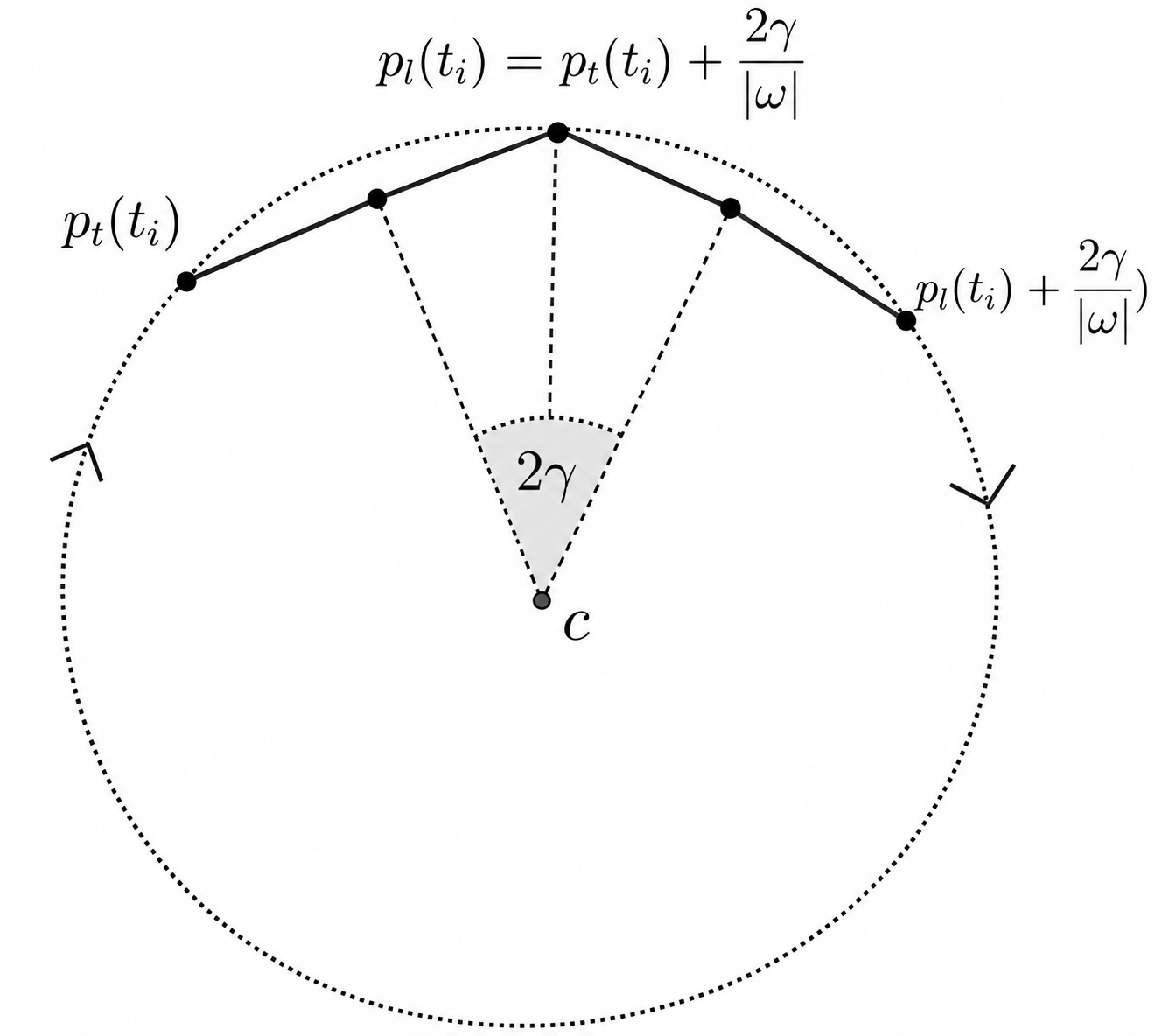}
    \quad
    \includegraphics[width=0.39\linewidth]{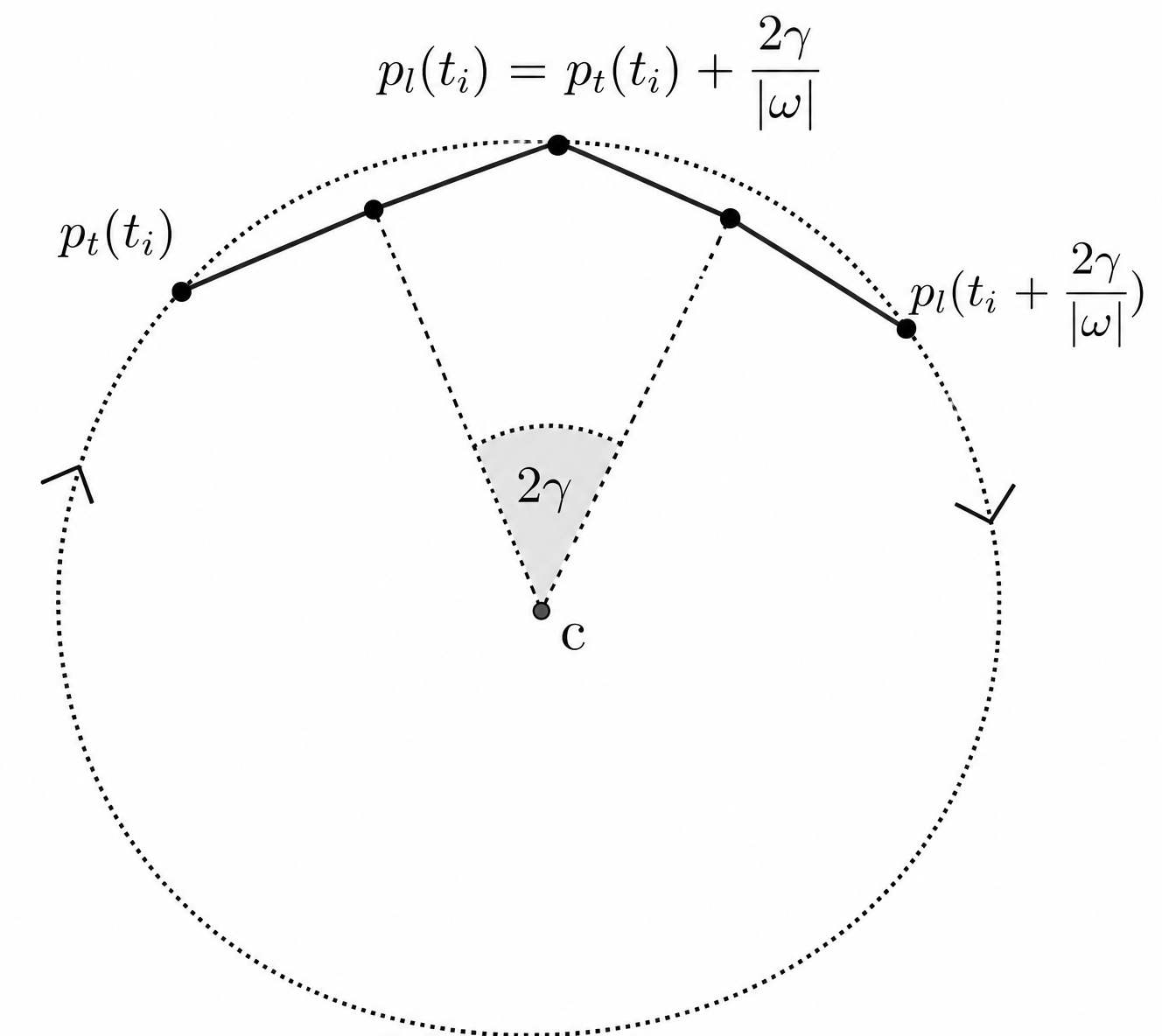}
    \caption{ The left image shows the definition of the angle $\gamma$. The right image illustrates the leading point.  }
    \label{fig:leading-point}
    \end{figure}

By Eq. \eqref{eq:lead-point}, $p_l(t) = p_f(t)$ and $p_t(t) = p_b(t)$ when $v>0$, whereas $p_l(t) = p_b(t)$ and $p_t(t) = p_f(t)$ when $v<0$. Furthermore, we say that an impact state $(q,\dot{q})$ is of \textbf{leading type} if the leading point lies on the impact surface, and of \textbf{trailing type} if the trailing point lies on the impact surface. It follows that the spaces of regular pre-impact and post-impact states of leading type in the generic case are defined by
\begin{equation}\label{eq:def-reg-led-im}
    \begin{split}
        \mathcal{I}_l^- & \coloneqq  \big\{ z \in \mathcal{I}_f^- : 0 <v \;\&\;\omega \neq 0\big\} \cup \big\{ z \in \mathcal{I}_b^- : v <0 \;\&\;\omega \neq 0 \big\},\\
        \mathcal{I}_l^+ & \coloneqq  \big\{ z \in \mathcal{I}_f^+ : 0<v \;\&\;\omega \neq 0 \big\} \cup \big\{ z \in \mathcal{I}_b^+ : v <0 \;\&\;\omega \neq 0 \big\}.\\
    \end{split}
\end{equation}
In addition, the spaces of regular pre-impact and post-impact states of trailing type in the generic case are defined by
\begin{equation}\label{eq:def-reg-tra-im}
    \begin{split}
        \mathcal{I}_t^- & \coloneqq  \big\{ z \in \mathcal{I}_f^- : v<0 \;\&\;\omega \neq 0 \big\} \cup \big\{ z \in \mathcal{I}_b^- : 0 <v  \;\&\;\omega \neq 0\big\},\\
        \mathcal{I}_t^+ & \coloneqq  \big\{ z \in \mathcal{I}_f^+ : v <0 \;\&\;\omega \neq 0 \big\} \cup \big\{ z \in \mathcal{I}_b^+ : 0 <v  \;\&\;\omega \neq 0 \big\}.
    \end{split}
\end{equation}
Therefore, the spaces of pre-impact states and post-impact states for the generic smooth dynamics are
$$ \mathcal{I}_g^-\coloneqq  \mathcal{I}_l^- \cup \mathcal{I}_t^- ,\;\;\text{and}\;\;  \mathcal{I}_g^+\coloneqq  \mathcal{I}_l^+\cup \mathcal{I}_t^+.$$

We begin by establishing the relationship between the $i$-th post-impact and the $(i+1)$-st pre-impact states for the leading-type and trailing-type cases.

          \begin{proposition} \label{lem:led-tra-rel}
        Let $z_i^+\in \mathcal{I}^+_g$ and $z_{i+1}^-\in\mathcal{I}^-_g$ denote the $i$-th post-impact state and the $(i+1)$-st pre-impact state, respectively. The following statements hold:
         \begin{itemize}
             \item If $z_i^+ \in \mathcal{I}^+_l$, then $z_{i+1}^-\in \mathcal{I}^-_t$.
             \item If $z_i^+\in\mathcal{I}^+_t$, then $z_{i+1}^-\in\mathcal{I}^-_l$.
         \end{itemize}
     \end{proposition}
     \begin{proof}
         Consider the smooth dynamics of the knife-edge with initial state $z_i^+$ at time $t_i$. As noted, the leading point of the knife-edge moves along a circle. The proof relies on the fact that the intersection of two circles consists of exactly two points when the circles intersect transversely and of a single point when they are tangent. With this observation in mind, we consider the leading and trailing cases separately.
         
         \textbf{(Leading-type case)} If $z_i^+\in \mathcal{I}^+_l$, then, by Remark \ref{rem:omega-sign}
         $$ p_t(t_i+\frac{2\gamma}{|\omega|}) = p_l(t_i).$$ 
         Thus, at time $t_i+\frac{2\gamma}{|\omega|}$, the trailing point is on the impact surface, and its velocity is directed into the disk. This contradicts the fact that the knife-edge remains in the disk during smooth dynamics. Hence, the trailing point must impact the surface before $t_i+\frac{2\gamma}{|\omega|}$, and the leading point cannot impact before the trailing point. 
         
         \textbf{(Trailing-type case)} If $z_i^+\in \mathcal{I}^+_t$, then the trace  $p_t((t_i,t_i+\Delta t])$ lies inside the disk exactly when $p_l(t_i+\Delta t)$ lies inside the disk. Indeed, suppose for contradiction that a segment of $p_t((t_i,t_i+\Delta t])$ lies outside the disk while $p_l(t_i+\Delta t)$ lies inside the disk. Then, by Remark \ref{rem:omega-sign}, the trace $p_t([t_i,t_i+\Delta t+\frac{2\gamma}{|\omega|}])$ would intersect the boundary of the disk three times if $z_i^+\in \mathcal{I}^+_t\cap \mathcal{RI}^+$ or two times if $z_i^+\in \mathcal{I}_t^+\cap \mathcal{TI}$, a contradiction. Therefore, the trailing point cannot impact the boundary before the leading point. 
     \end{proof}

    \begin{lemma}\label{prop:gen-rot-map}
   Let  $z_i^+\in \I_g^+$ and $z_{i+1}^-\in \I_g^-$ denote the $i$-th post-impact state and the $(i+1)$-st pre-impact state, respectively. Then the identities in Eqs. \eqref{eq:gen-indet-1} and \eqref{eq:gen-indet-2} hold.
    \end{lemma}

            \begin{figure}
  \centering
   \includegraphics[width=.48\linewidth]{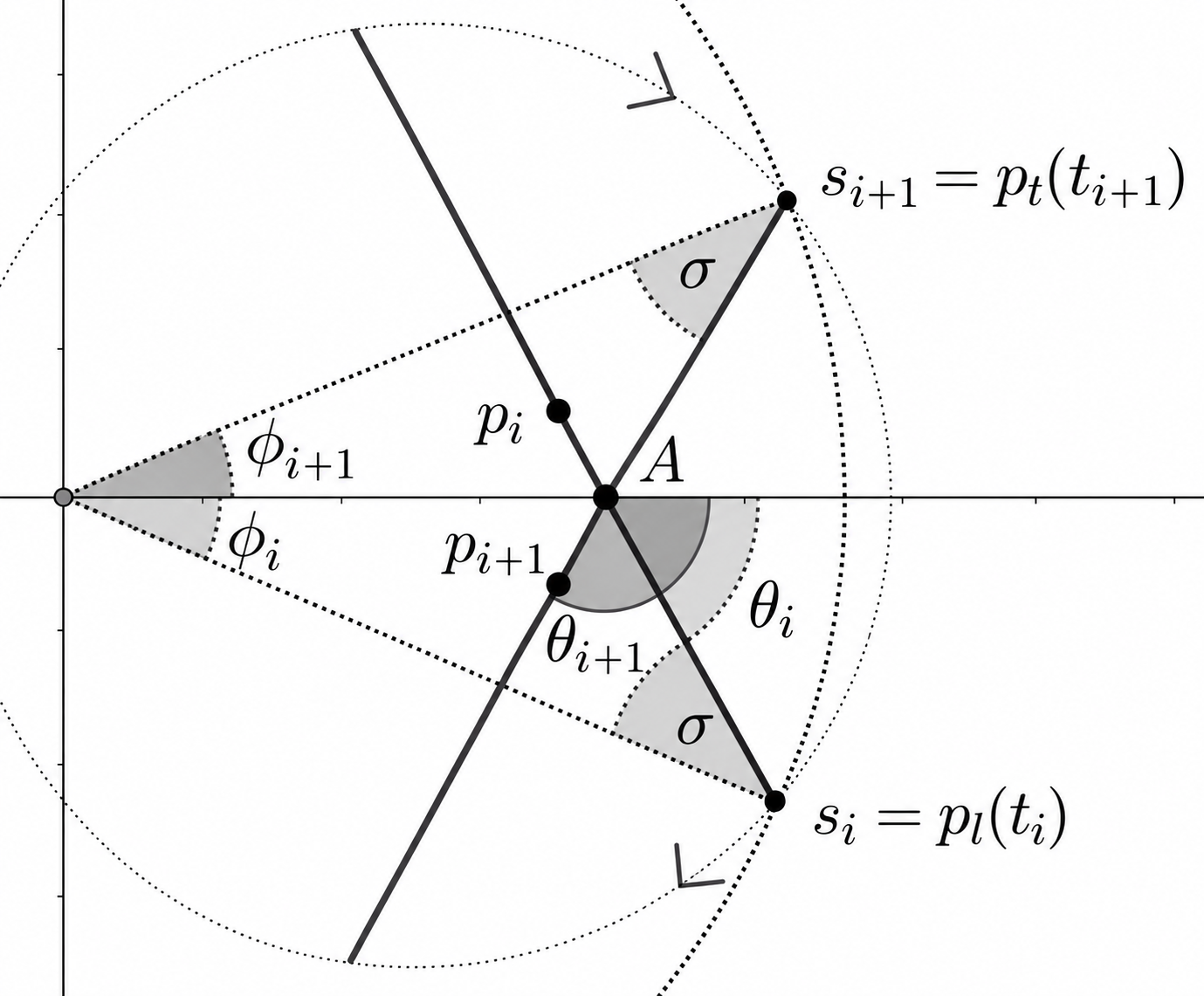}
  \includegraphics[width=.35\linewidth]{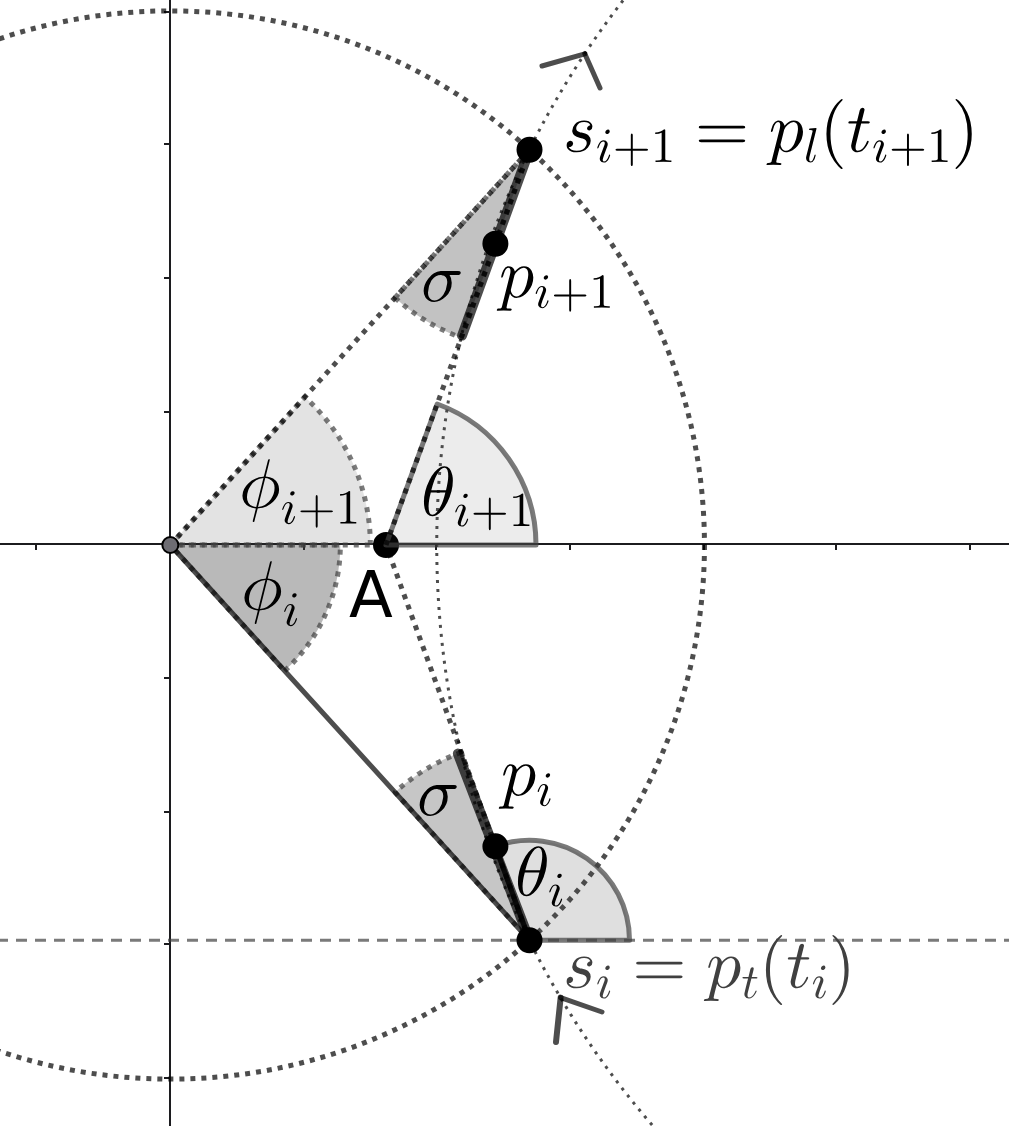}
    \caption{ The two panels illustrate the angle relations used in the proof of Eq \eqref{eq:gen-indet-2}. The left panel corresponds to the case $z^+_i\in \I_l^+\cap\mathcal{RI}^+$, while the right panel corresponds to the case $z^+_i\in \I_t^+\cap\mathcal{RI}^+$. }
    \label{fig:circle-3}
    \end{figure}

     \begin{proof}
     Consider the initial state $z_i^+\in \I_g^+$. Then the smooth dynamics describes circular motion centered at $\mathrm{c}(z_i^+)$. Using the rotational symmetry of the system, we may assume, without loss of generality, that $\mathrm{c}(z_i^+)$ lies on the positive $x$-axis, as shown in Fig. \ref{fig:circle-1}.

    Recall from Remark \ref{rem:omega-sign} that the circular motion is clockwise when $\omega_i<0$ and counterclockwise when $\omega_i>0$. After applying this rotation, the post-impact state $z_i^+$ has the following form:
    \begin{equation*}
            z_i^+ = \begin{cases}
                (-\alpha_i,\vartheta_i,\omega_i,v_i), \quad & \text{if}\; \omega_i<0,\\
                (\alpha_i,\vartheta_i,\omega_i,v_i) , \quad \;\;\;&\text{if}\; 0<\omega_i. 
            \end{cases}
    \end{equation*}
    With this setup, we first prove Eq. \eqref{eq:gen-indet-1} for
    $ z_{i}^+ \in \I_g^+$ with $\omega_i<0$. By construction, the impact points $s_i$ and $s_{i+1}$ in Fig. \ref{fig:circle-3} correspond to the states $z^+_i$ and $z_{i+1}^-$, respectively. From this construction, we  have $\phi_i = -\alpha_i$ and $\phi_{i+1} = \alpha_i$. Hence, 
    $$\phi_{i+1} = \phi_i+2\alpha_i = \phi_i-2\sgn(\omega_i)\alpha_i ,$$
    since $\sgn(\omega_i) = -1$. An analogous argument applies when $0<\omega_i$.

    To prove Eq. \eqref{eq:gen-indet-2}, we use an argument based on similar triangles, treating the cases $z_{i}^+ \in \mathcal{I}_l^+$ and $z_{i}^+ \in \mathcal{I}_t^+$ separately.

     \textbf{(Leading-type case)} We distinguish between the cases $z^+_i\in \I_l^+\cap\mathcal{RI}^+$ and $z^+_i\in \I_l^+\cap\mathcal{TI}$. First, consider a regular post-impact state $z_i^+\in\I_l^+\cap \mathcal{RI}^+$, and assume that $\omega_i<0$. By construction, we may choose a representative $\phi_i \in (-\pi,0)$. To determine the orientation of the knife-edge, we work in the original coordinate $\theta$. First suppose that $\theta_i$ admits a representative in $(-\pi,0)$. 
     
     On the one hand, the angles of the triangle with vertices at the origin, $s_i$ and $A$ in Fig. \ref{fig:circle-3} satisfy
    $$-\phi_i+(\pi+\theta_i) + \sigma = \pi.$$
    On the other hand, the angles of the triangle with vertices at the origin, $s_{i+1}$ and $A$ in Fig. \ref{fig:circle-3} satisfy
    $$ \phi_{i+1} - \theta_{i+1} + \sigma = \pi.   $$
    Subtracting these identities eliminates $\sigma$. Using the change of coordinates $\vartheta = \phi-\theta$, we then obtain
    $$ \vartheta_{i+1} = -\vartheta_i +\pi. $$ 
    The same argument applies when $\theta_i \in (0,\pi)$.  The case $\theta_i = \pi$ follows from the relations $\phi_i = -\alpha_i$, $\phi_{i+1} = \alpha_i$, and $\theta_{i+1} = 0$.  Indeed, 
     $$ \phi_{i+1}-\theta_{i+1} = \alpha_i  = \alpha_i +2 \pi =  \theta_i- \phi_i +\pi. $$
     Using the change of coordinates $\vartheta = \phi-\theta$ yields the desired identity.

     In Fig. \ref{fig:circle-3}, we assume that the trailing point is back-type. The symmetry in Eq. \eqref{eq:fron-back-sym} yields the identity when the leading point is front-type. An analogous argument applies when $0<\omega_i$.

     We now consider the case $z_i^+\in\I_l^+\cap \mathcal{TI}$, and assume that $\omega_i<0$. In this case $\phi_i = \phi_{i+1} = 0$, and Fig. \ref{fig:circle-tan-1} shows that 
     $$ -\theta_{i+1} = \theta_i+\pi.  $$
     As in the previous case, the change of coordinates $\vartheta = \phi-\theta$ yields the desired identity. The case $0<\omega_i$ is proved in the same way.

    \textbf{(Trailing-type case)} We distinguish between the cases $z^+_i\in \I_t^+\cap\mathcal{RI}^+$ and $z^+_i\in \I_t^+\cap\mathcal{TI}$. First, consider a regular post-impact state $z_i^+\in\I_t^+\cap\mathcal{RI}^+$, and assume that $\omega_i<0$. As before, we choose a representative $\phi_i \in (-\pi,0)$, and use $\theta$ to determine the orientation of the knife-edge. Assume that $\theta_i$ admits a representative in $(0,\pi)$. 
    
    On the one hand, the angles of the triangle with vertices at the origin, $s_i$ and $A$ in Fig. \ref{fig:circle-3} satisfy
    $$-\phi_i + \theta_i + \sigma  = \pi.$$
    On the other hand, the angles of the triangle with vertices at the origin, $s_{i+1}$ and $A$ in Fig. \ref{fig:circle-3} satisfy
    $$ \phi_{i+1} + (\pi - \theta_{i+1}) + \sigma = \pi.   $$
    Subtracting these identities eliminates $\sigma$. Using the change of coordinates $\vartheta = \phi-\theta$, we obtain
    $$ \vartheta_{i+1} = -\vartheta_i -\pi. $$ 
    We handle the remaining cases analogously to the corresponding cases in the leading-type argument. In particular, the case $z^+_i\in \I_t^+\cap\mathcal{TI}$ is illustrated in Fig. \ref{fig:circle-tan-1}.
    \end{proof} 

    \begin{figure}
  \centering
   \includegraphics[width=.43\linewidth]{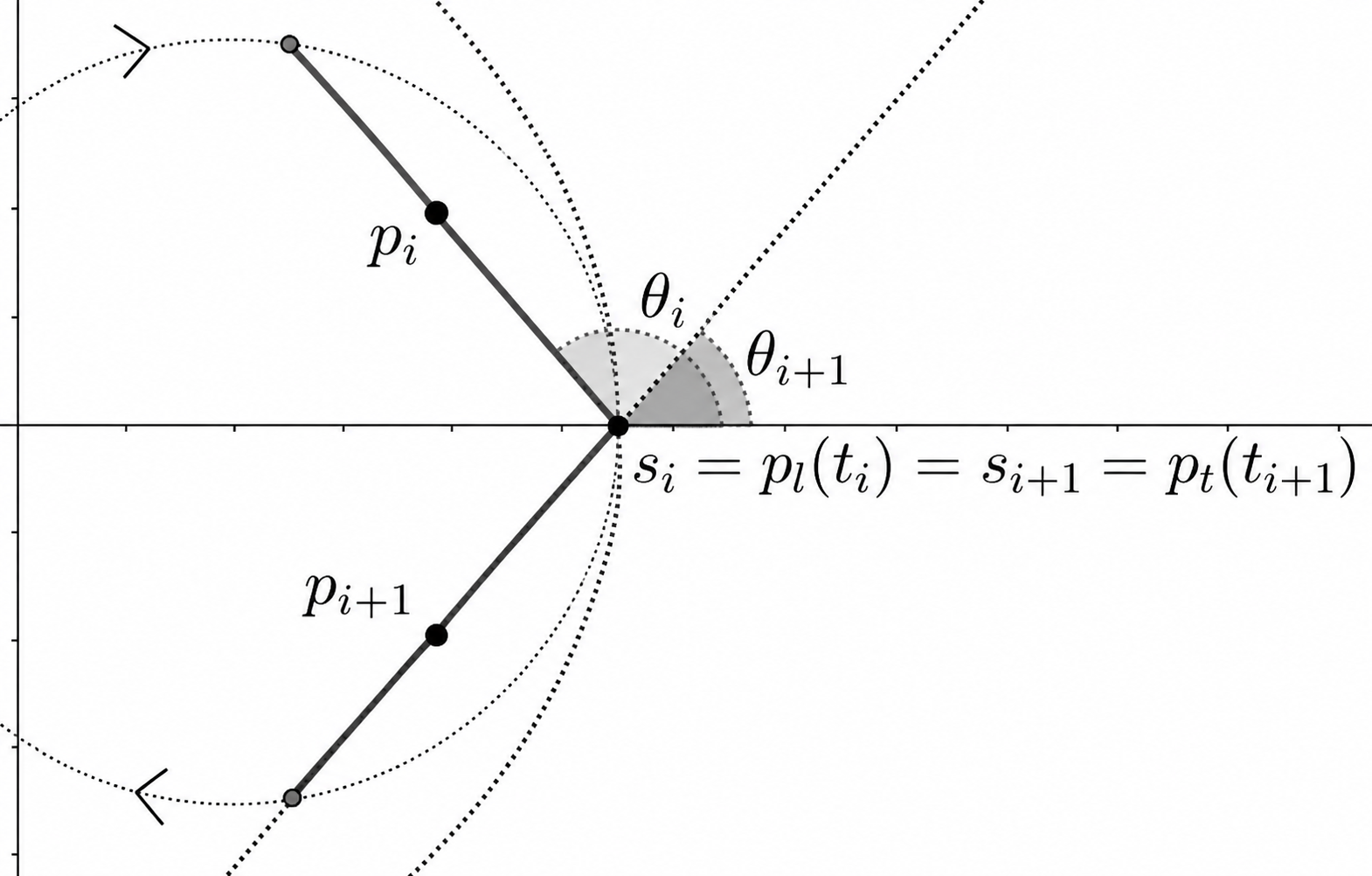}
  \includegraphics[width=.43\linewidth]{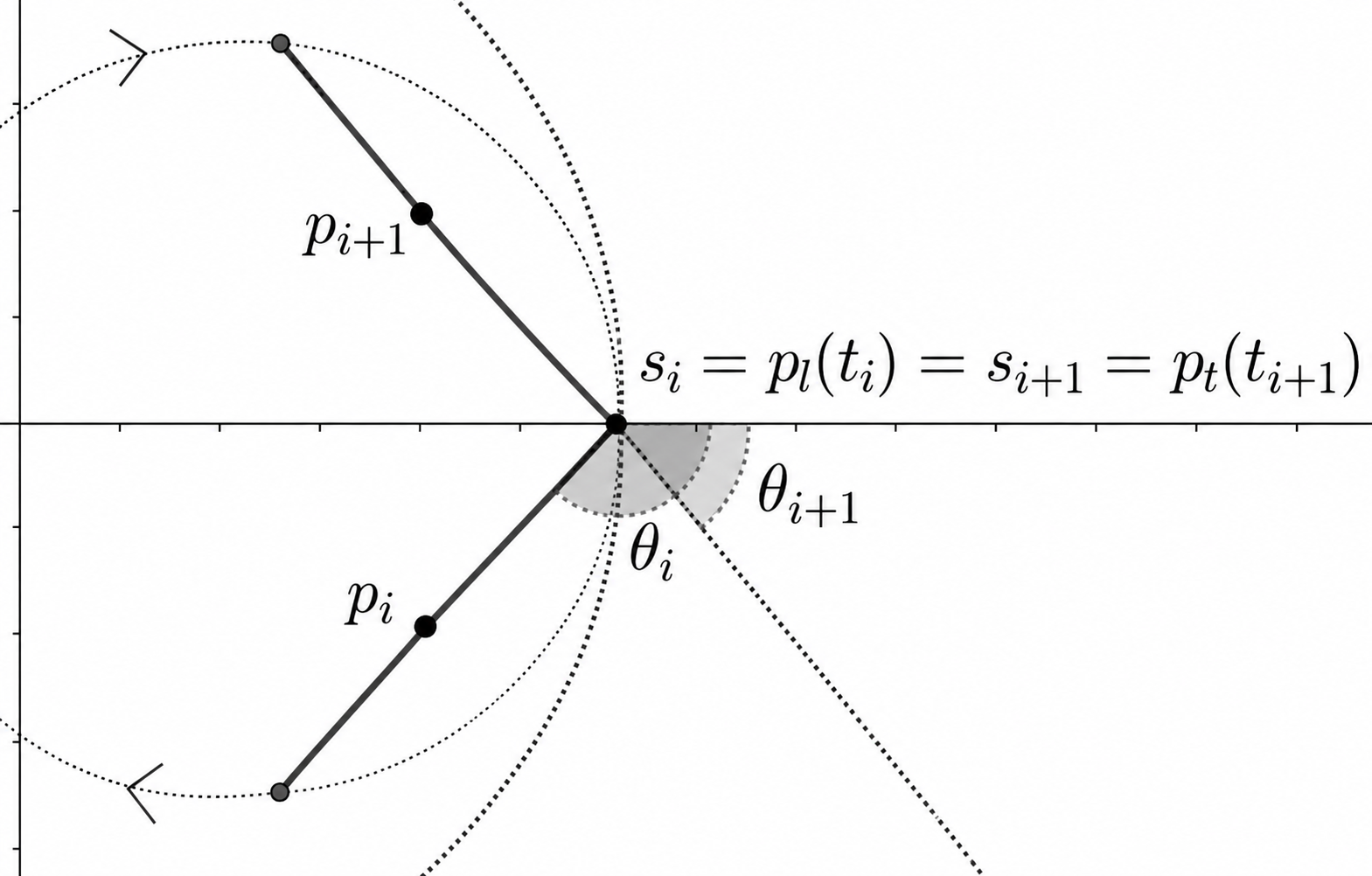}
    \caption{ The two panels illustrate the angle relations used in the proof of Eq. \eqref{eq:gen-indet-2}. The left panel corresponds to the case $z^+_i\in \I_l^+\cap\mathcal{TI}$, while the right panel corresponds to the case $z^+_i\in \I_t^+\cap\mathcal{TI}$. }
    \label{fig:circle-tan-1}
    \end{figure}

     Up to this point, only the angle $\alpha$ has been used. To compute the elapsed time $\Delta t_i$ between consecutive impacts, we also require the angles $\beta$ and $\gamma$. We define the generic time-of-flight function $\Delta t_g: \mathcal{I}^+_g\to (0,\infty)$ by
     $$  \Delta t_g(z^+) = \begin{cases}
     \frac{2(\beta(z^+)+\gamma(z^+)-\pi)}{|\omega|}, &\text{if}\; z^+ \in \I_l^+, \\
         \frac{2(\beta(z^+)-\gamma(z^+))}{|\omega|}, \;&\text{if}\; z^+ \in \I_t^+. \\
     \end{cases}  $$ 

     \begin{lemma}\label{lem:del-t-generic}
        If $z_i^+\in \I_g^+$, then $\Delta t_i :=\Delta t_g(z_i^+)$ is the elapsed time between the $i$-th and the $(i+1)$-st impacts.
     \end{lemma}
     \begin{proof}
          For convenience, let $\beta_i:=\beta(z_i^+)$ and $\gamma_i:=\gamma(z_i^+)$. Since $\dot{\theta} = \omega_i$ along the smooth trajectory between the $i$-th and $(i+1)$-st impacts, define $\Delta \theta_i:=\theta_{i+1}-\theta_i$. Therefore, $|\Delta \theta_i| = |\omega_i|\Delta t_i$, and it suffices to compute $|\Delta \theta_i|$ in the cases $ z_i^+ \in \I_l^+$ and  $z_i^+ \in \I_t^+$.

         By rotational symmetry, we may assume, without loss of generality, that $\mathrm{c}(z_i^+)$ lies on the positive $x$-axis.

         First, consider  $ z_i^+ \in \I_l^+$. The midpoints $p_i$ and $p_{i+1}$ in Fig. \ref{fig:circle-2} correspond to the states $z_i^+$ and $z_{i+1}^-$, respectively. From this construction, the midpoint rotates through an unoriented angle $2(\beta_i+\gamma_i-\pi)$, as illustrated in Fig. \ref{fig:circle-2}. Hence, 
         $$|\Delta \theta_i| =  2(\beta_i+\gamma_i-\pi).$$

         We now consider the case $ z_i^+ \in \I_t^+$. Then the midpoints $p_i$ and $p_{i+1}$ in Fig. \ref{fig:circle-2} correspond to the states $z_i^+$ and $z_{i+1}^-$, respectively. From this construction, the midpoint rotates through an unoriented angle $2(\beta_i-\gamma_i)$, as illustrated in Fig. \ref{fig:circle-2}. Hence, $|\Delta\theta_{i}| = 2(\beta_i-\gamma_i)$.  
     \end{proof}

      \begin{remark}\label{rem:int-ext-int}
        The orientation of the arc connecting $p_i$ and $p_{i+1}$ characterizes the geometric distinction between leading-type and trailing-type post-impact states.  For a leading-type post-impact state, this arc faces the boundary of the disk, whereas for a trailing-type post-impact state, it faces the center of the disk; see Fig. \ref{fig:circle-2}. 
        
        This distinction determines the nature of the intersection between the smooth trajectory and the caustic curve appearing in Theorem \ref{thm:2}: the intersection is internally tangent in the leading-type case and externally tangent in the trailing-type case.
    \end{remark}

    \begin{figure}
    \centering
    \includegraphics[width=0.37\linewidth]{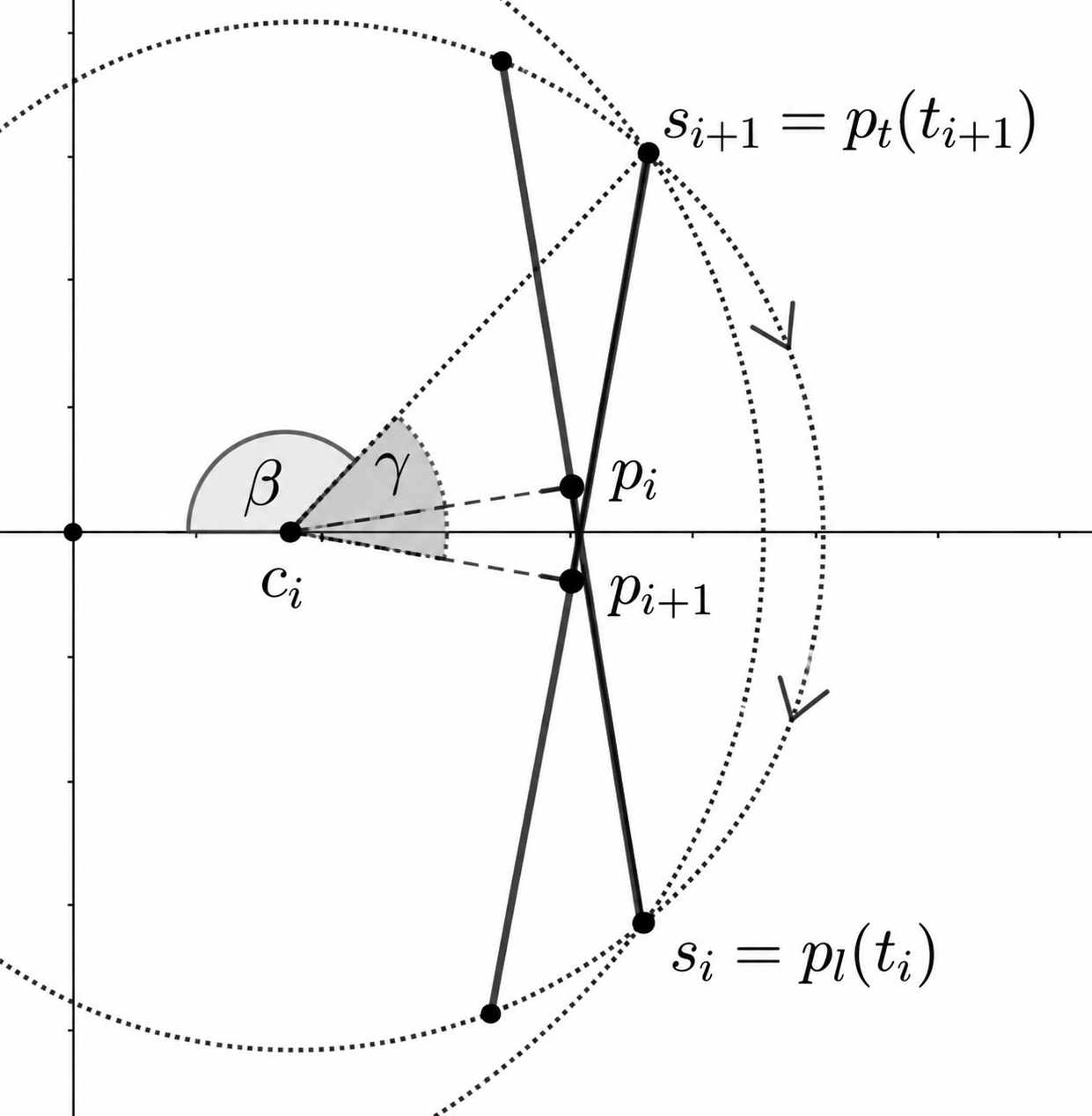}
    \quad
    \includegraphics[width=0.35\linewidth]{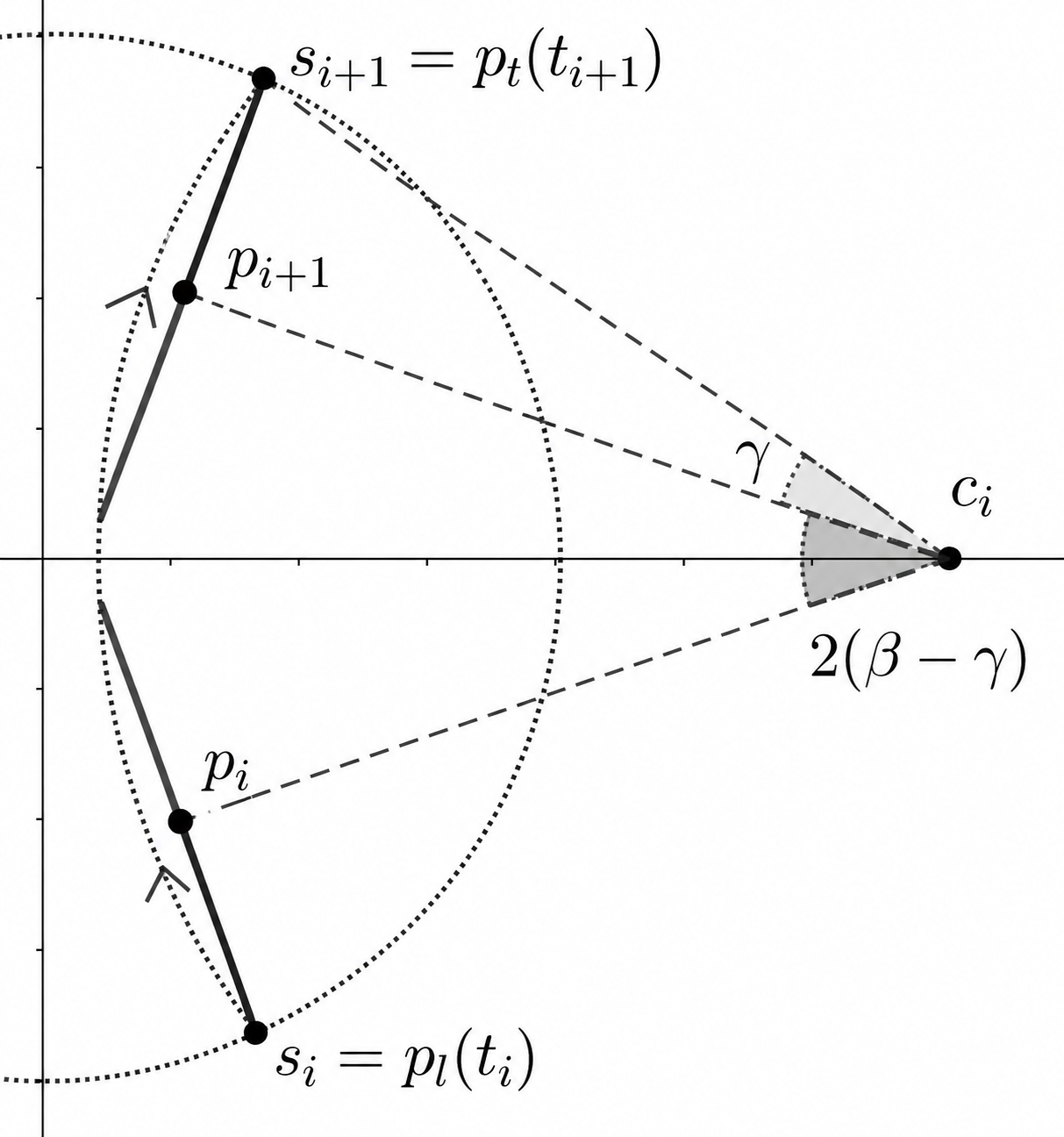}
    \caption{ Both panels illustrate the impact points $s_i$ and $s_{i+1}$, together with the arc of the midpoint trajectory between $p_i:=p(t_i)$ and $p_{i+1} := p(t_{i+1})$. In the first figure, the post-impact state $z_i^+$ is leading-type, and the corresponding angular displacement is $|\Delta \theta| = 2(\beta_i+\gamma_i-\pi)$. In the second figure, the post-impact state $z_i^+$ is trailing-type, and the angular displacement is $ |\Delta \theta| = 2(\beta_i-\gamma_i)$.  }
    \label{fig:circle-2}
    \end{figure}

\subsection{Singular Case \texorpdfstring{\(\omega \neq 0, v =0\)}{} }\label{sub-sec:imp-st-sing-ome-neq-zer}

In this case, the kinetic energy is purely rotational, and the notion of leading and trailing points is lost. We therefore use the front and back points to classify the impact states.

 Let $z^+\in \mathcal{I}^+$ be a post-impact state with $0 \neq \omega $ and $0=v$. Then the solution to the nonholonomic equations of motion with initial state $z^+$ is given by
\begin{equation*}\label{eq:mot-v-eq-zero}
    \theta(t)  = \omega t +\theta, \quad p(t) = \mathrm{c}(z^+).
\end{equation*}
Therefore, the midpoint $p(t)$ of the knife-edge is at rest, and the rigid body rotates about its midpoint. Consequently, the front and back points move along circles centered at  $\mathrm{c}(z^+)$ with radius $r(z^+) = \ell$.

For the singular case $\omega \neq  0$, $v=0$, define the pre-impact and post-impact state spaces by
\begin{equation}\label{eq:def-sin-im-om-neq-zer}
    \begin{split}
        \mathcal{I}_\omega^- & \coloneqq \big\{ z \in \mathcal{I}^- : v = 0 \big\},\\
        \mathcal{I}_\omega^+ &\coloneqq  \big\{ z \in \mathcal{I}^+ : v = 0 \big\}.
    \end{split}
\end{equation}
The space of tangential impact states with  singular smooth dynamics $\omega \neq 0$ is characterized as 
$$ \mathcal{TI}_\omega =\big\{ z\in \I: v= 0 \;\&\; \vartheta=\pm \frac{\pi}{2}\big\}.  $$ 
Note that the above condition is equivalent to the radial condition $\|\mathrm{c}(z^+)\| = R-\ell$.

\begin{proposition} \label{lem:fro-back-rel}
         Let $z_i^+$ and $z_{i+1}^-$ denote the $i$-th post-impact state and the $(i+1)$-st pre-impact state, respectively. The following relations hold:
         \begin{itemize}
             \item If $z_i^+ \in \mathcal{I}^+_\omega$ is front-type, then $z_{i+1}^-\in\mathcal{I}^-_\omega$ is  back-type.
             \item If $z_i^+\in \mathcal{I}^+_\omega$ is  back-type, then $z_{i+1}^-\in \mathcal{I}^-_\omega$ is  front-type.
         \end{itemize}
     \end{proposition}

     \begin{proof}
         Suppose first that $z_i^+\in \mathcal{I}^+_\omega$ is front-type, and consider the smooth dynamics with initial state $z^+_i$  at time $t_i$; that is $p_f(t_i) \in \Si$. Using the identity $p_b(t_i+\frac{\pi}{|\omega|})=p_f(t_i)$, we conclude that at time $t_i+\frac{\pi}{|\omega|}$ the back point is on the impact surface. If $z_i^+\in \mathcal{I}^+_\omega \cap \mathcal{RI}^+$, then the back point lies on the impact surface with velocity directed inward. This contradicts the assumption that the knife-edge remains inside the disk during smooth dynamics, so the back point impacts the disk boundary before time $t_i+\frac{\pi}{|\omega|}$. If $z_i^+\in \mathcal{I}^+_\omega \cap \mathcal{TI}$, then the limiting state as $t\to (t_i+\frac{\pi}{|\omega|})^-$ is a back-type tangential impact state. The proof of the back-type case is analogous, with the roles of the front and back points interchanged.
     \end{proof}

    \begin{lemma}\label{prop:sin-v-zero-rot-map}
   Let  $z_i^+\in \I_\omega^+$ and $z_{i+1}^-\in \I_\omega^-$ denote the $i$-th post-impact state and the $(i+1)$-st pre-impact state, respectively. Then the identities in Eqs. \eqref{eq:gen-indet-1} and \eqref{eq:gen-indet-2} hold.
     \end{lemma}
     \begin{proof}
         Since $v=0$, Eq. \eqref{eq:def-cent} implies that $\mathrm{c}(z^+_i) = p_i$. By rotational symmetry, we may therefore assume, without loss of generality, that $\mathrm{c}(z^+_i) = p_i$ lies on the positive $x$-axis. The result then follows as in the proof of Lemma \ref{prop:gen-rot-map}; see Fig. \ref{fig:circle-4}. 
     \end{proof}
     
     \begin{figure}
    \centering
    \includegraphics[width=0.33\linewidth]{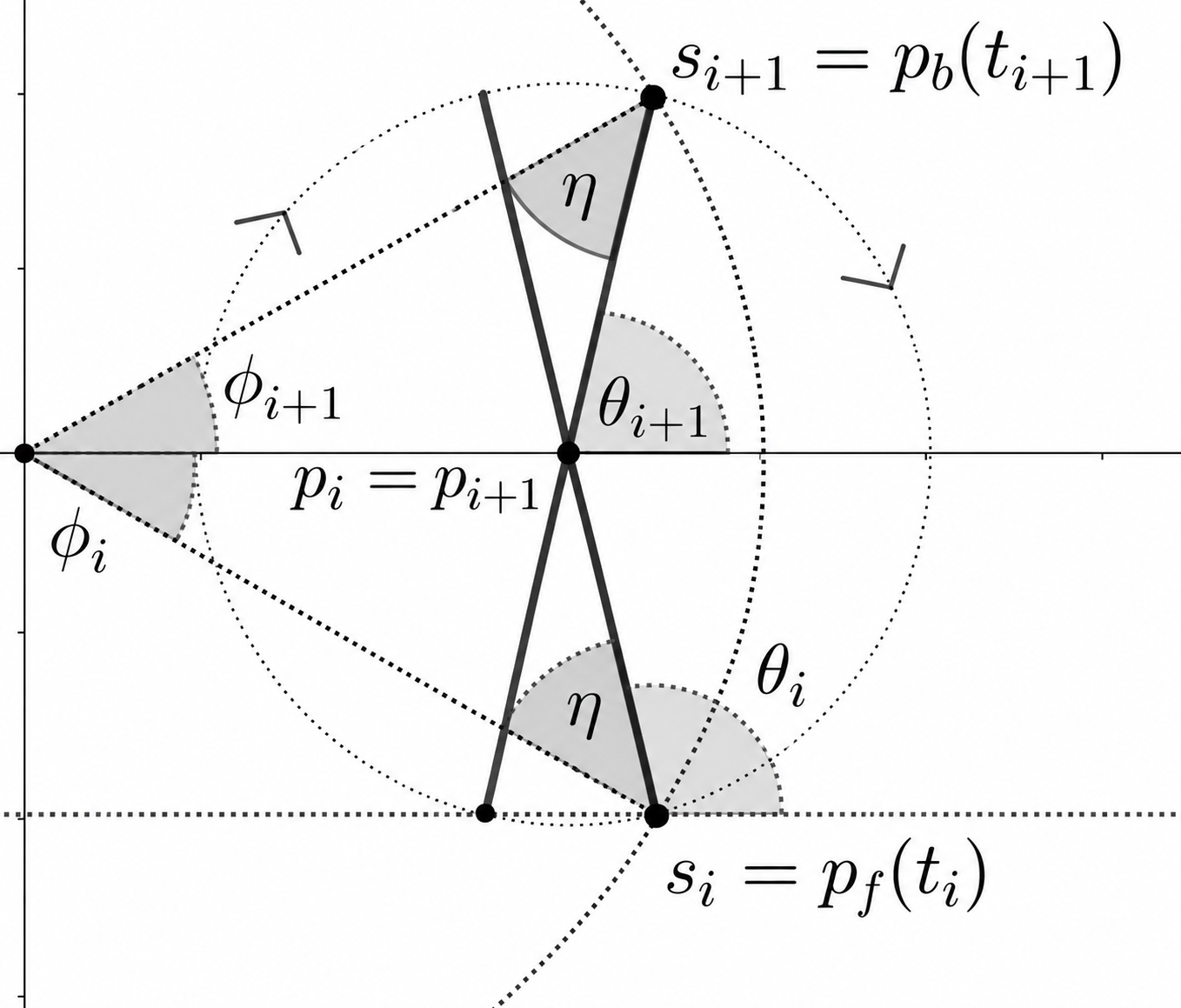}
    \quad
    \includegraphics[width=0.33\linewidth]{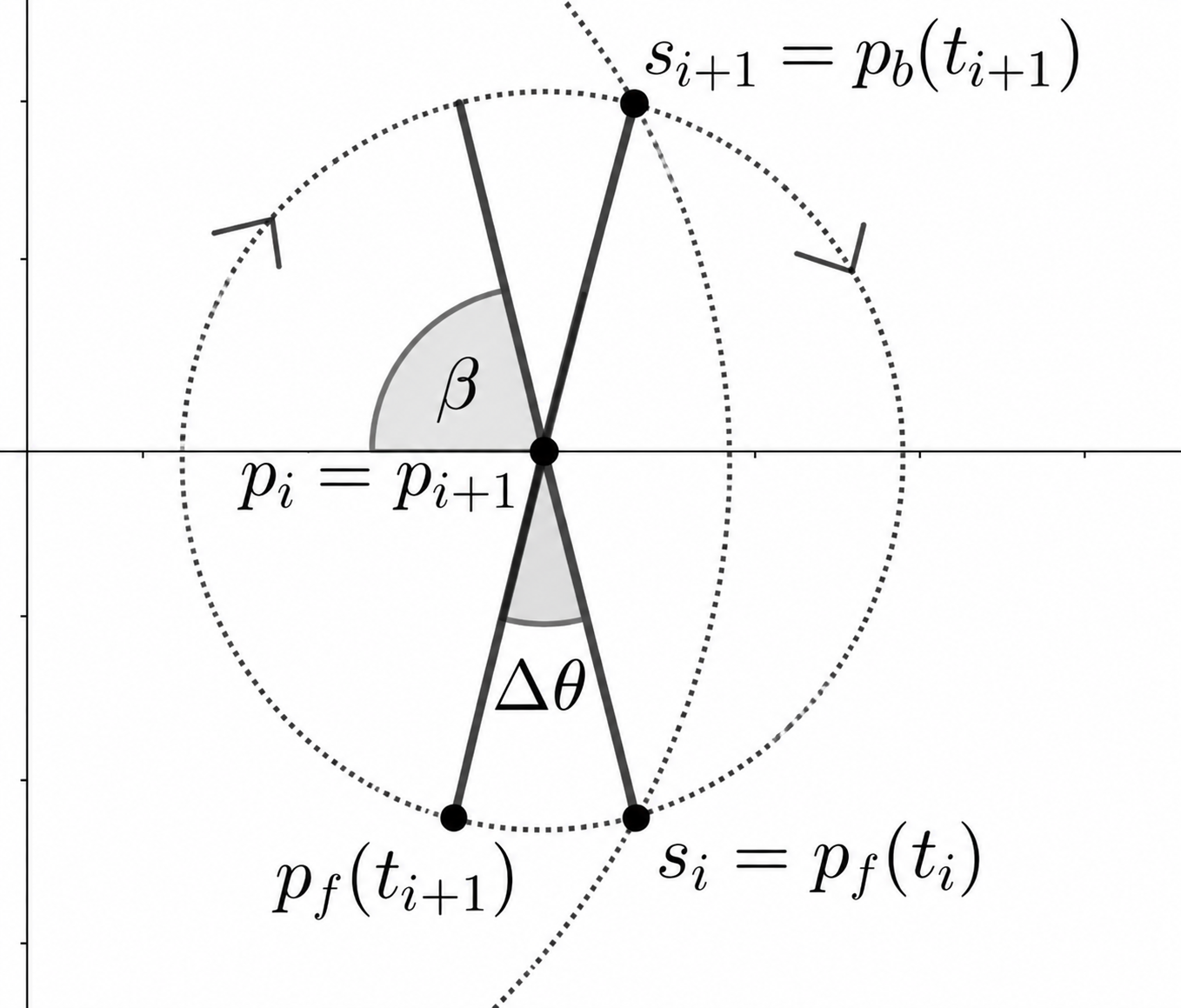}
    \caption{ Both panels illustrate the impact points $s_i$ and $s_{i+1}$. The left figure shows the essential angles used to prove Lemma \ref{prop:sin-v-zero-rot-map} for the singular case $\omega \neq 0$. The right figure shows that $ \Delta \theta_i = 2\beta_i-\pi$.  }
    \label{fig:circle-4}
    \end{figure}

    We now define the time-of-flight function $\Delta t_\omega: \mathcal{I}^+_\omega\to (0,\infty)$ by
     $$  \Delta t_\omega(z^+) = \frac{2\beta(z^+)-\pi}{|\omega|} . $$

     \begin{lemma}\label{lem:del-t-omega}
     If $z_i^+\in \I_\omega^+$, then  $\Delta t_\omega(z_i^+)$ is the elapsed time between the $i$-th and the $(i+1)$-st impacts.
     \end{lemma}
     \begin{proof}
          Consider the initial state $z_i^+\in \I_\omega^+$. Then the midpoint position $p_i=p_{i+1}$ in Fig. \ref{fig:circle-4} corresponds to both states  $z_i^+$ and $z_{i+1}^-$. From this construction, the front point rotates through an unoriented angle $2\beta_i-\pi$, as shown in Fig. \ref{fig:circle-4}. Hence, $|\Delta\theta_i| = 2\beta_i-\pi$.
     \end{proof}

\subsection{Singular Case \texorpdfstring{\(\omega =0, v \neq  0\)}{}}\label{subsec:smo-sing-dy}
In this case, the kinetic energy is purely translational, and the definition of the leading and trailing points given in Eq. \eqref{eq:lead-point} remains valid.  In contrast to the generic case, the pre-impact states are always trailing-type, and the post-impact states are always leading-type.

Let $z^+\in \mathcal{I}^+$ be a post-impact state with $0= \omega $ and $0\neq v$. Then the solution to the nonholonomic equations of motion with initial state $z^+$ is given by
\begin{equation*}\label{eq:mot-ome-eq-zero}
\theta(t)  = \theta(z^+), \quad p(t) = p(z^+)+ vt (\cos\theta, \sin\theta). 
\end{equation*}
Thus, the midpoint of the knife-edge and the leading point move along a straight line. 

The space of post-impact states for the singular smooth dynamics $v \neq  0$ was defined in Eq. \eqref{eq:def-sin-im-v-neq-zer}. We now define the corresponding space of pre-impact states:
\begin{equation*}
        \mathcal{RI}_v^-  \coloneqq   \big\{ z \in \mathcal{RI}^- : \omega = 0\big\} .
\end{equation*}


Before continuing, we record the following observation about the spaces $\mathcal{RI}_v^-$ and $\mathcal{RI}_v^+$.
\begin{remark}\label{rem:angl-rep-ome-zer} Eqs. \eqref{eq:def-fro-po} and \eqref{eq:def-bac-po} imply that, for each impact state, we can take a representative of the equivalence class $[\vartheta]$ as follows.
\begin{itemize}
    \item If $z^-\in \mathcal{RI}_v^-$ with $0<v$, then $\vartheta\in (-\frac{\pi}{2},\frac{\pi}{2})$, whereas if $z^-\in \mathcal{RI}_v^-$ with $v<0$, then $\vartheta \in (\frac{\pi}{2},\frac{3\pi}{2})$.

    \item If $z^+\in \mathcal{RI}_v^+$ with $0<v$, then $\vartheta\in (\frac{\pi}{2},\frac{3\pi}{2})$, whereas if $z^+\in \mathcal{RI}_v^+$  with $v<0$, then $\vartheta \in (-\frac{\pi}{2},\frac{\pi}{2})$.
\end{itemize}
\end{remark}

\begin{lemma}\label{prop:sin-vneq-zer-rot-map}
     Let  $z_i^+\in \mathcal{RI}_v^+$ and $z_{i+1}^-\in \mathcal{RI}_v^-$ denote the $i$-th post-impact state and the $(i+1)$-st pre-impact state, respectively. Then the identities in Eqs.  \eqref{eq:gen-indet-1} and \eqref{eq:gen-indet-2} hold.
\end{lemma}
Before presenting the proof, we remark that, in this case, the system resembles the dynamics of a classical billiard in a disk \cite[Chapter 2]{tabachnikov2005geometry}.
\begin{proof}[Proof of Lemma \ref{prop:sin-vneq-zer-rot-map}] 
    Let us begin with the identity \eqref{eq:gen-indet-1}.
    It is well known that, for the classical billiard in a disk, the angle defining the rotation map is twice the angle between the post-impact velocity and the tangent vector to the circle at the impact point. In this case, we consider the angle $\sigma$ between $\dot{p}(t)$ and the tangent vector to the circle; see Fig. \ref{fig:v-zero-angles}. We compute $\sigma$ as follows: if $z_i^+ \in \mathcal{RI}_v^+$ with $0<v$ and $\langle\cdot,\cdot\rangle:\R^2\times\R^2\to \R$ is the Euclidean inner product, then $\sigma$ is determined by the equation
    $$ \cos\sigma = \langle (\cos\theta_i,\sin\theta_i), (-\sin\phi_i,\cos\phi_i) \rangle  =\cos\big(-\vartheta_i-\frac{\pi}{2}\big).  $$
     It follows from Remark \ref{rem:angl-rep-ome-zer} that there exists a representative $\vartheta_i$ such that $-\vartheta_i-\frac{\pi}{2} \in (0,\pi)$. Thus, we can use the principal branch of the $\arccos$ function to find that the desired angle is $\sigma = -\vartheta_i-\frac{\pi}{2}$. Therefore, the angle defining the impact point $s_{i+1}$ is given by
     $$\phi_{i+1} = \phi_i - 2\vartheta_i -\pi.$$ 
     An analogous argument applies when $v<0$. In this case, the angle defining the impact point $s_{i+1}$ is given by
     $$\phi_{i+1} = \phi_i - 2\vartheta_i +\pi. $$ 
     It follows from the two equations above that $[\phi_{i+1}] = [\phi_i] +[-2\vartheta_i+\pi]$. Hence,  if $z_i^+ \in \mathcal{RI}_v^+$, then $[\phi_{i+1}] = [\phi_i] +\Omega(z^+_i)$.

     We continue by proving the identity \eqref{eq:gen-indet-2}. Since $\omega_i = 0$ and $\omega$ is constant along the smooth dynamics, we have $\omega = 0$ and hence $\dot{\theta} = 0$ between the $i$-th and $(i+1)$-st impacts. Thus, $\theta_i=\theta_{i+1}$, and 
     $$[\vartheta_{i+1}] = [\phi_{i+1}-\theta_{i+1}]  = [\phi_i-\theta_i] +[-2\vartheta_i+\pi] = [-\vartheta_i+\pi].$$  
\end{proof}

\begin{figure}
    \centering
    \includegraphics[width=0.35\linewidth]{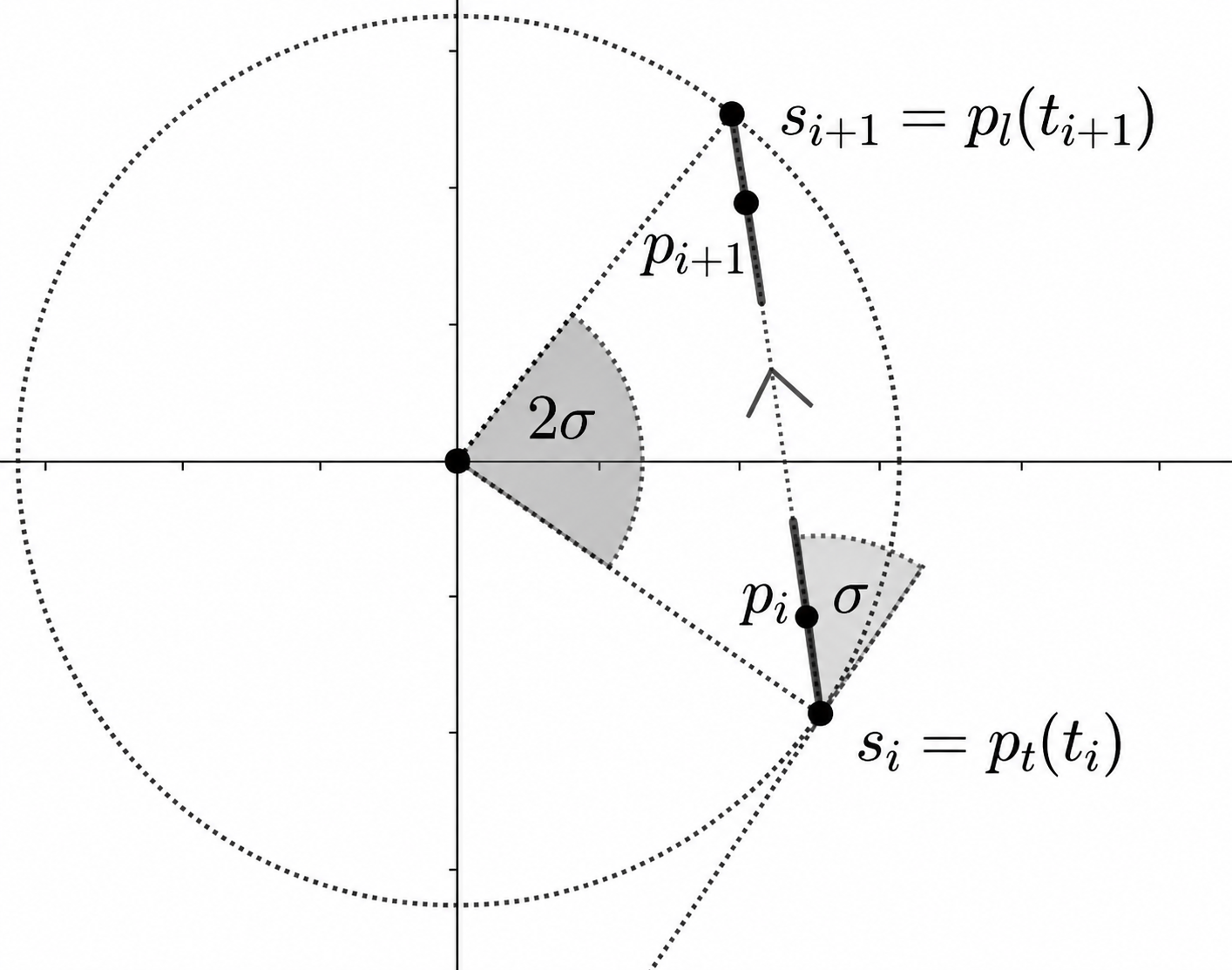}
    \caption{The figure illustrates the analogy between the classical billiard in a disk and the case  $v\neq 0$, together with the angle $\sigma$.  }
    \label{fig:v-zero-angles}
    \end{figure}

    We define the time-of-flight function $\Delta t_v: \mathcal{RI}^+_v\to (0,\infty)$ by
    $$  \Delta t_v (z^+) = \frac{2(|R\cos\vartheta| - \ell)}{|v|} . $$

     \begin{lemma}\label{lem:del-t-v}
      If $z_i^+\in \mathcal{RI}_v^+$, then  $\Delta t_v(z_i^+)$ is the elapsed time between the $i$-th and $(i+1)$-st impacts.
     \end{lemma}

     \begin{proof}
     As shown in Fig. \ref{fig:v-zero-angles}, the base of the isosceles triangle whose vertices are the origin and the two impact points has length
         $$ 2R|\sin\sigma| = 2R|\cos\vartheta|.  $$
         Subtracting the length of the knife-edge from the length of this base yields the distance traveled by the midpoint and hence the desired result. 
     \end{proof}

We conclude this subsection by proving that the function $\Omega:\I^+ \to \mathbb{S}^1$ is continuous. 
     \begin{lemma}
         The function $\Omega:\I^+ \to \mathbb{S}^1$ defined in Eq. \eqref{eq:def-rot-map} is continuous.  
     \end{lemma}
     \begin{proof}
         By the piecewise definition of $\Omega:\I^+ \to \mathbb{S}^1$, it suffices to consider a sequence $z^+_n\in \I^+ \setminus\mathcal{RI}^+_v$ and show that if $z^+_n \to z^+ \in \mathcal{RI}^+_v$ as $n\to \infty$, then 
         $$ [-2\sgn(\omega_n) \alpha(z^+_n)] \to [- 2\vartheta + \pi], \qquad \text{as}\qquad n\to \infty. $$

          For brevity, set $d_n = d(z^+_n)$ and $r_n = r(z^+_n)$. From Eqs. \eqref{eq:def-al-bet}, \eqref{eq:def-cent}, and \eqref{eq:def-rad-mot}, it follows that
         \begin{equation*}
             \frac{d_n^2+R^2-r_n^2}{2d_nR} = \frac{\sgn(\omega_n)\big( v_n\sin \vartheta_n + \omega_n(R+\ell\cos\vartheta_n)\big)}{\sqrt{v_n^2 + \omega^2_n(R^2+\ell^2+2\ell\cos\vartheta_n) + 2R\omega_nv_n\sin\vartheta_n}}.
         \end{equation*}
         After passing to a subsequence, we may assume that $\omega_n>0$. Under this assumption, we obtain
         \begin{equation*}
             \lim_{n\to \infty} \frac{d_n^2+R^2-r_n^2}{2d_nR}  = \sgn(v) \sin\vartheta.
         \end{equation*}
         By continuity of $\arccos$, we obtain
         \begin{equation*}
             \lim_{n\to \infty} -2\arccos\big( \frac{d_n^2+R^2-r_n^2}{2d_nR} \big) = \begin{cases}
                 -2\vartheta + \pi &\;\text{if} \; 0 <v,\\
                -2\vartheta - \pi &\;\text{if} \; v<0.\\
             \end{cases}
         \end{equation*}
         An analogous argument works for $\omega_n<0$.  In either case, interpreting the limiting expression modulo $2\pi$, we obtain $  \Omega(z^+_n) \to \Omega(z^+)$ along every subsequence on which $\omega_n$ has constant sign. Since every subsequence of $\{z_n^+\}$ admits a further subsequence with this property, it follows that $  \Omega(z^+_n) \to \Omega(z^+)$ as $n\to \infty$.
     \end{proof}

\subsection{Proof of Theorem \ref{the:rot-map}}\label{subsec:thm-1}

\begin{proof}[Proof of Theorem \ref{the:rot-map}] 
Lemmas \ref{prop:gen-rot-map}, \ref{prop:sin-v-zero-rot-map}, and \ref{prop:sin-vneq-zer-rot-map} imply that 
$$ z^-_{i+1} = \mathrm{Ro}(z_i^+). $$

Define $\Delta t : \I^+ \to (0,\infty)$  by
\begin{equation}\label{eq:fin-Del-t}
    \Delta t(z^+) = \begin{cases}
        \Delta t_g(z^+) &\;\text{if}\;z^+ \in \I^+_g,\\
        \Delta t_\omega(z^+) &\;\text{if}\;z^+ \in \I^+_\omega,\\
        \Delta t_v(z^+) &\;\text{if}\;z^+ \in \I^+_v.\\
    \end{cases}
\end{equation}
Thus, Lemmas \ref{lem:del-t-generic}, \ref{lem:del-t-omega}, and  \ref{lem:del-t-v} imply that if $z^+_{i} \in \I^+$, then $\Delta t(z^+_i)$ is the elapsed time between consecutive impacts.
\end{proof}


\section{Proof of Theorem \ref{mainthe:slei}}\label{sec:proof-main-the}

In this section, we prove Theorem \ref{mainthe:slei}. To this end, we first compute the impact map $\mathrm{I}:\I^-\to \I^+$ in the $z$-coordinates and hence obtain the map  $\T\coloneqq \mathrm{I}\circ\mathrm{Ro}.$
Since the expression for $\mathrm{I}(z^-)$ depends on whether $z^-\in \I^-_f$ or $z^-\in \I^-_b$, we first define the corresponding reflection maps, following the discussion in Subsection \ref{subsec:qua-vel}. Let $\mathrm{D}:\mathbb{S}^1 \to \R$ be the function  defined by
    $$ \mathrm{D}(\vartheta)\coloneqq  \frac{\ell^2}{J}\sin^2(\vartheta)+ \frac{1}{m}\cos^2(\vartheta).  $$
    For each $\vartheta$, define the reflections $\Rm_f^\vartheta:\R^2\to\R^2$ and $\Rm_b^\vartheta:\R^2\to\R^2$ by
    \begin{equation*}
\Rm_f^\vartheta (\omega,v ) \coloneqq   \begin{pmatrix}
                   \omega,v
                   \end{pmatrix}
    \begin{pmatrix}
        1-\frac{2\ell^2\sin^2\vartheta}{J\mathrm{D}(\vartheta)} & -\frac{2 \ell\cos\vartheta\sin\vartheta}{m\mathrm{D}(\vartheta)}  \\
       -\frac{2\ell\cos\vartheta\sin\vartheta}{J\mathrm{D}(\vartheta)}  &  1-\frac{2\cos^2\vartheta}{m \mathrm{D}(\vartheta)}\\
    \end{pmatrix}
\end{equation*}
and 
\begin{equation*}
\Rm_b^\vartheta (\omega,v ) \coloneqq   \begin{pmatrix}
                   \omega,v
                   \end{pmatrix}
    \begin{pmatrix}
        1-\frac{2\ell^2\sin^2\vartheta}{J\mathrm{D}(\vartheta)} & \frac{2\ell\cos\vartheta\sin\vartheta}{m\mathrm{D}(\vartheta)}  \\
       \frac{2 \ell\cos\vartheta\sin\vartheta}{J\mathrm{D}(\vartheta)}  &  1-\frac{2\cos^2\vartheta}{m \mathrm{D}(\vartheta)}\\
    \end{pmatrix}.
\end{equation*}

\begin{lemma}\label{lem:impct-map}
     The \textbf{impact map} $\mathrm{I}:\mathcal{I}^-\to \mathcal{I}^+$ is given by
\begin{equation}\label{eq:imp-map-sle}
    \mathrm{I}(z^-) = \begin{cases}
    
    \big(\phi,\vartheta,\Rm_f^\vartheta (\omega,v )\big),&\;\text{if}\; z^-\in \mathcal{I}_f^-, \\
    \big(\phi,\vartheta,\Rm_b^\vartheta (\omega,v )\big),&\;\text{if}\; z^-\in \mathcal{I}_b^- .
    \end{cases}
\end{equation}

\end{lemma}

\begin{proof}
    We begin with the front-type case. Consider the intrinsic definition of the function $h:SE(2)\to \R$ given in Eq. \eqref{eq:round table-1}, expressed in the coordinates $z$. If $z^-\in \I^-_f$, Eqs. \eqref{eq:h-dif-fro} and \eqref{eq:hor-gra} yield
    $$ \mathrm{D}(\vartheta) = g(\horg h,\horg h). $$
    Therefore, using the orthonormal frame $\{\frac{1}{\sqrt{J}}X_1,\frac{1}{\sqrt{m}}X_2\}$ and substituting Eq. \eqref{eq:h-dif-fro} and the expression above into Eq. \eqref{eq:imp-map-qua}, we obtain the desired formula for $\Rm_f^\vartheta$. The back-type case is analogous and yields the stated formula for $\Rm_b^\vartheta$, completing the proof.
\end{proof}


\subsection{Proof of Theorem \ref{mainthe:slei}}\label{subsec:prof-main-the}
We now prove our second main theorem. 

\begin{proof}[Proof of Theorem \ref{mainthe:slei}]
     Consider $z \in \mathcal{I}^+_f$. Since Corollary \ref{def:cor-rot-map} implies that  $\T(z) \in \mathcal{I}^+_b$, we obtain
     $$ \T^2(z) = \big(\phi+\varphi(z), \vartheta, (\mathcal{R}_b^{-\vartheta+\pi}\circ \mathcal{R}^\vartheta_f)(\omega,v )\big),$$
     where $\varphi:\I^+\to \mathbb{S}^1$ is the function defined  in Eq. \eqref{eq:def-varphi-fuc}.
     

     Moreover, the expressions for the reflections $\mathcal{R}^\vartheta_f,\mathcal{R}^\vartheta_b:\R^2\to\R^2$, given in Lemma \ref{lem:impct-map}, imply that $\mathcal{R}_b^{-\vartheta+\pi} \equiv \mathcal{R}^\vartheta_f$. Since a reflection is an involution, we have 
     $$\mathcal{R}_b^{-\vartheta+\pi}\circ \mathcal{R}^\vartheta_f \equiv id.$$
     This identity proves the desired result for the case $z \in \mathcal{I}^+_f$. An analogous argument applies to the case $z \in \mathcal{I}^+_b$. 
\end{proof}

\subsection{Proof of Theorem \ref{thm:2}}

We begin this section by defining the function $\mathrm{f}:\I^+ \to [0,R) $, which determines the family of caustics in Theorem \ref{thm:2}. 

\begin{equation}\label{eq:func-f-def}
    \mathrm{f}(z^+) \coloneqq  \begin{cases}
        |d(z^+) - |\frac{v}{\omega} | | \;&\;\text{if} \;\; z^+ \in \I_g^+,\\
        R|\sin\vartheta| \;&\;\text{if} \;\; z^+ \in \mathcal{RI}_v^+,\\
        |d(z^+) - \ell | \;&\;\text{if} \;\; z^+ \in \I_\omega^+.
    \end{cases}
\end{equation}
Here, $d:\mathcal{I}^+\setminus\mathcal{RI}_v^+ \to (0,\infty)$ is again the function defined by Eq. \eqref{eq:d-for-fron-type}.

We are now ready to prove Theorem \ref{thm:2}.

\begin{proof}[Proof of Theorem \ref{thm:2}] 
We consider the same cases as in the definition of  $\mathrm{f}(z^+)$.

\textbf{(Case $z_0^+ \in \I^+_g)~$} Let $z_0^+ \in \I^+_g$ and consider the smooth dynamics with initial state $z_{2i}^+ = \T^{2i}(z^+_0)$. By Theorem \ref{mainthe:slei},  during each smooth segment, the midpoint moves along a circle of radius $|\frac{v_0}{\omega_0} |$ and center $\mathrm{c}_{2i}$ defined by Eq. \eqref{eq:def-cent}. We claim that the distance $d_{2i}$ from the origin to the center $\mathrm{c}_{2i}$ is independent of $i$. Indeed, Proposition \ref{lem:led-tra-rel} implies that if $z_0^+$ is front-type, then $z_{2i}^+$ is also front-type. Eq. \eqref{eq:d-for-fron-type} together with Theorem \ref{mainthe:slei} implies that $d_{0}=d_{2i}$. 

Recall the following fact from Euclidean geometry. Consider two circles with radii  $r_1$ and $r_2$ whose centers are separated by a distance $d$. Then the circles are externally tangent if and only if $d=r_1+r_2$, and internally tangent if and only if $d=|r_1-r_2|$.   In our case, the distance between the centers is $d_0 $  and the circles have radii $|\frac{v_0}{\omega_0} |$ and  $|d_0 - |\frac{v_0}{\omega_0}||$. Therefore, the circles are externally tangent if $0 < d_0 - |\frac{v_0}{\omega_0}|$, and internally tangent if $ d_0 - |\frac{v_0}{\omega_0}|<0$; see Remark \ref{rem:int-ext-int}.

\textbf{(Case $z_0^+ \in \mathcal{RI}^+_v)~$} Let $z_0^+ \in \mathcal{RI}^+_v$, and consider the smooth dynamics with initial state $z_{2i}^+ = \T^{2i}(z^+_0)$. By Theorem \ref{mainthe:slei}, on each smooth segment, the midpoint follows a linear trajectory whose associated chord subtends an angle  $2\sigma$; see Fig. \ref{fig:v-zero-angles}. Therefore, the distance from the origin to the midpoint of the chord is $R|\cos(-\vartheta+\frac{\pi}{2})|$, which yields the desired formula.  

Finally, the case $ z_0^+ \in \I^+_\omega$ follows the same strategy as the case $z_0^+ \in \I^+_g$.
\end{proof}

\subsubsection{Invariant Spaces}

We define the subset of perpendicular impact states by
$$ \I^\perp \coloneqq  \{z\in \I: \vartheta = \pi k, \;k\in \mathbb{Z} \}.\footnote{The condition $\vartheta = \pi k$ with $k\in \mathbb{Z}$ implies that the normal vector to the disk boundary is parallel to the orientation of the knife-edge. Hence, the knife-edge is perpendicular to the tangent vector to the disk boundary.}  $$

The following lemma identifies $\T$-invariant subsets of the impact space. 

\begin{lemma}\label{lem:perp-imp}
The subset $\mathcal{TI}$ and $\I^\perp\cap\I^+$ are $\T$-invariant.
\end{lemma}

\begin{proof}
    The space of tangential impact states is $\T$-invariant by definition. 
    Let $z^+ \in \I^\perp\cap\I^+$. By Eq. \eqref{eq:def-rot-map}, $\mathrm{Ro}(z^+) \in \I^\perp \cap \I^-$. Since $\mathcal{R}_b^{\vartheta}(\omega,v ) = (\omega,-v)$ and $\mathcal{R}^\vartheta_f(\omega,v ) = (\omega,-v)$ whenever $\vartheta = \pi k$ for some $k\in \mathbb{Z}$, it follows that $\T(z^+) \in \I^\perp\cap\I^+$. Thus, $\I^\perp\cap\I^+$ is $\T$-invariant.
\end{proof}

We conclude this section by observing that the circles $\mathcal{C}^0(z^+)$ and $\mathcal{C}^1(z^+)$ coincide when  $z^+ \in \mathcal{TI}$ or $z^+ \in \I^\perp\cap\I^+$, as shown in Fig. \ref{fig:per-den-mot}. Moreover, when $z^+ \in \I^+_\omega$, one of the caustic circles coincides with the circle containing the midpoint positions at the impact states, as shown in Fig. \ref{fig:v-zero}.

\section*{Conclusion and Future Work}

We presented a formula for elastic nonholonomic impacts in terms of the horizontal gradient. We then introduced the knife-edge billiard in a disk and parametrized its impact states. Using rotational symmetry and elementary Euclidean geometry, we constructed a map  $\T:\I^+\to \I^+$ that describes the hybrid dynamics through the relation
$$z^+_{i+1} = \T(z_i^+).$$
We showed that the system is integrable in the sense that its long-term dynamics can be explicitly determined from the post-impact state $z_i^+$ and the recurrence relation above. This example adds another integrable billiard system in a smooth, strictly convex, bounded domain to the known examples, which include the classical billiard in a disk, the classical elliptic billiard, and the magnetic billiard in a disk.

Several natural directions remain for future work. One direction is the Chaplygin sleigh billiard in a disk \cite{borisov2009dynamics,bloch1996nonholonomic}. This generalizes the knife-edge case to situations where the center of mass does not coincide with the blade contact point.
In this case, the nonholonomic smooth dynamics can be integrated in terms of elliptic functions. The Chaplygin sleigh is also one of the simplest mechanical systems exhibiting the dissipative character of nonholonomic dynamics.
 Another direction is the knife-edge billiard in an elliptical domain, where the rotational symmetry used in the present work is broken; see Fig. \ref{fig:future-work}. Numerical analysis of this case was carried out in \cite{clark2019bouncing}. This setting would require new techniques and may provide insight into the persistence or destruction of integrable structures in nonholonomic billiards.

\begin{figure}
    \centering
    \includegraphics[width=0.4\linewidth]{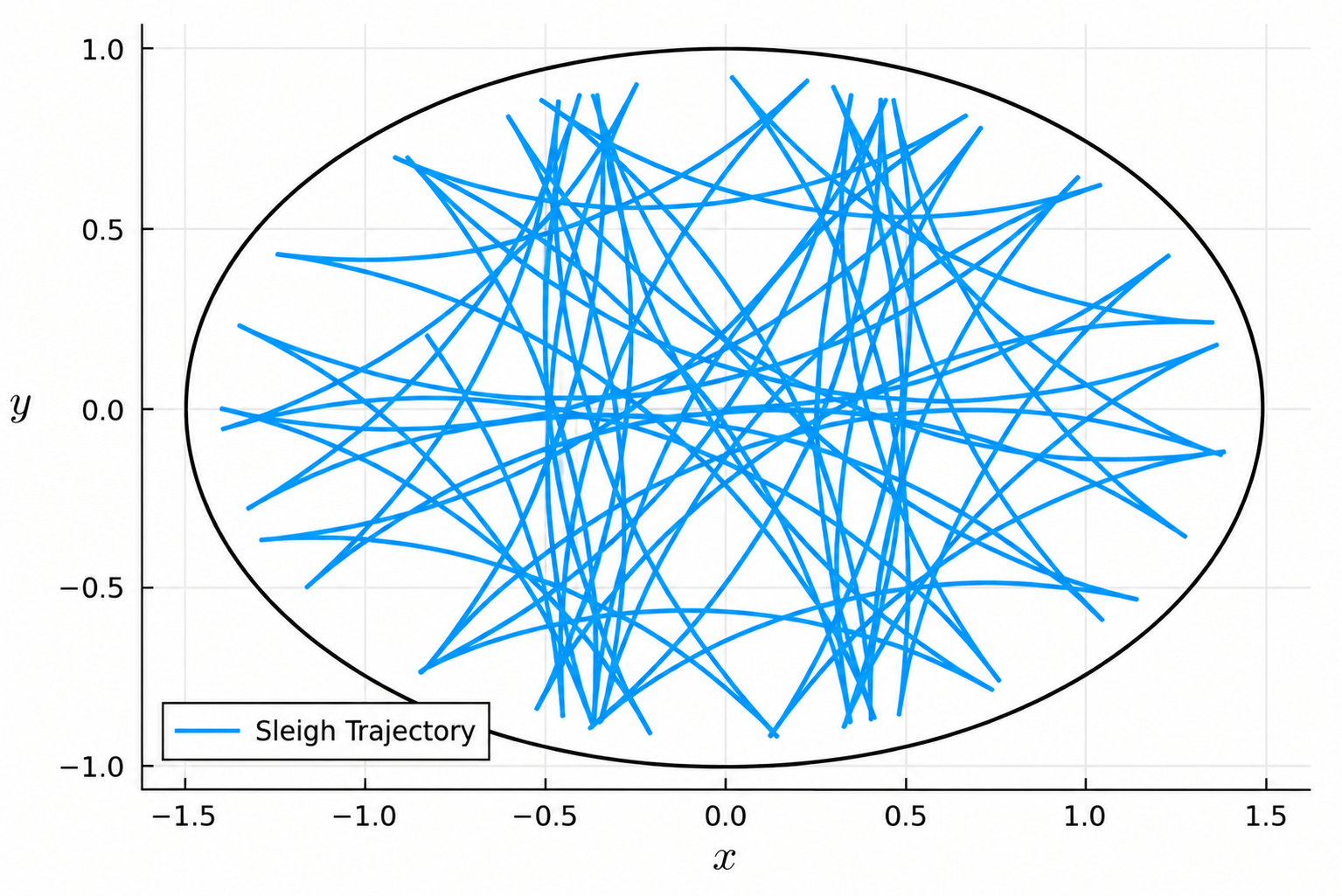}\quad \includegraphics[width=0.4\linewidth]{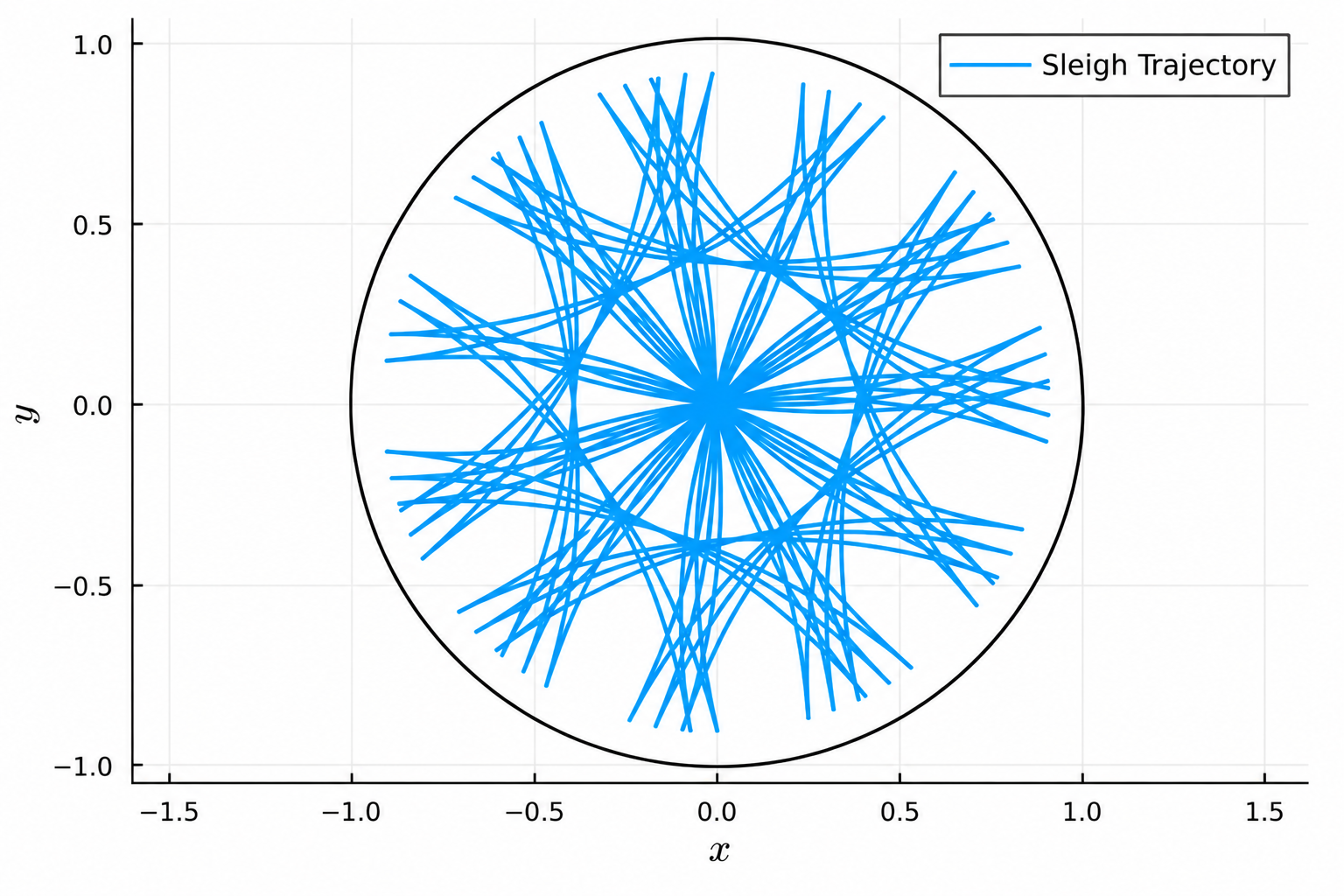}
    \caption{The left panel shows a Chaplygin sleigh billiard in a disk, while the right panel shows a knife-edge billiard in an elliptical table.  }
    \label{fig:future-work}
\end{figure}

\appendix

\section{Elastic Impact Map}\label{apd:imap-map}

    We recall that in Subsection \ref{ap:hyb-sys}, we assume that the Lagrangian function $\La:TQ\to\R$ is natural. The fiber derivative agrees with the musical isomorphism $\flat$:
    \begin{equation*}
        \begin{split}
            \flat:&  TQ \to T^*Q,\quad v \to g(v,\cdot),\\
            \sharp:& T^*Q \to TQ^,\quad \sharp = \flat^{-1}.
        \end{split}
    \end{equation*}

During this proof, we use Einstein summation notation: when an expression has a matching lower and upper index, a sum is implied, i.e., $a_kb^k = \sum_{k}a_kb^k$. 

\begin{proof}
    The spatial variations $\delta q$ must satisfy the nonholonomic constraints $\Theta^k(\delta q) = 0$ and the impact constraint $dh(\delta q) = 0$. Therefore, the impact map must satisfy
    \begin{equation}\label{eq:impact-map-cons}
        \begin{split}
            \big( FL^+-FL^-\big) & = \lambda_k \Theta^k+\alpha dh,\\
            \Delta H & = 0,\\
            \Theta^k(\dot{q}^+) & = 0.
        \end{split}
    \end{equation}
    Applying the musical isomorphism $\sharp:T^*Q\to TQ$ to the first equation in Eq. \eqref{eq:impact-map-cons}, we obtain $\dot q^+ = \dot q^-+\lambda_kW^k+\alpha \nabla h$, where $W^k = \sharp(\Theta^k)$ and $\nabla h = \sharp(dh)$. Since $\dot q^+,\dot q^- \in \D$, it follows that
    \begin{equation}\label{eq:alp-valu}
        0 =\lambda_k a^{k \ell} + \alpha \Theta^\ell (\nabla h),\;\text{for}\;\ell=1,\dots,m,\;\text{where}\;\;a^{k\ell} = \Theta^k(W^\ell). 
    \end{equation}
    Since the $\{\Theta^k\}_{k=1}^m$ are linearly independent, the matrix $a^{k\ell}$ is invertible. Denote its inverse by $a_{k\ell}$. It follows that the above system of equations has a solution
    $$ \lambda_k(\alpha) = - \alpha a_{k\ell}\Theta^\ell(\nabla h).$$
    Since the Lagrangian is natural, the kinetic energy can be expressed as $\frac{1}{2}p(\dot{q})$. The second equation in Eq. \eqref{eq:impact-map-cons} implies that the kinetic energies $p^+(\dot q^+)$ and $p^-(\dot q^-)$ are equal. This yields
    \begin{equation}\label{eq:kin-ener}
        \begin{split}
            0 =2\alpha dh(\dot q^-) +  \lambda_k\lambda_i a^{ki}+2 \lambda_k \alpha \Theta^k(\nabla h) + \alpha^2 dh(\nabla h),
        \end{split}
    \end{equation}
    where we used that $\Theta^k(\dot q^-) = 0$, $dh(\dot q^-) = p^-(\nabla h)$,  and $\Theta^k(\nabla h) = dh(W^k)$. Substituting the value of $\lambda_k(\alpha)$ found in Eq. \eqref{eq:alp-valu} into Eq. \eqref{eq:kin-ener}, we obtain
    \begin{equation*}
        \begin{split}
            0 & =\alpha \Big( 2 dh(\dot q^-) +\alpha \big(  dh(\nabla h) -  a_{k \ell} \Theta^\ell(\nabla h) \Theta^k(\nabla h)  \big) \Big)\\
            & =\alpha \Big( 2 dh(\dot q^-) +\alpha  g\big(\nabla h, \nabla h -a_{k\ell}\Theta^k(\nabla h) W^\ell\big)\Big).
        \end{split}
    \end{equation*}
    In the above equation, the term $\nabla h -a_{k\ell}\Theta^k(\nabla h) W^\ell$ is precisely the horizontal gradient $\horg h$, since the orthogonal projection $\pi_\D:TQ \to \D$ is given by 
    $$\pi_{\D}(\dot{q}) = \dot{q} - a_{k\ell}\Theta^k(\dot q) W^\ell.$$  The nontrivial solution $\alpha = \frac{-2dh(\dot{q}^-)}{g(\horg h,\horg h)}$ yields the desired result, since $$g(\horg h,\horg h) = g(\nabla h,\horg h).$$
\end{proof}

\nocite{*} 
\bibliographystyle{plain}
\bibliography{bibli}

\end{document}